\documentclass{article}
\usepackage[utf8]{inputenc}
\usepackage{amsmath,amsthm,amssymb,amsfonts}
\usepackage{ upgreek }
\usepackage{comment}

\theoremstyle{definition}
\newtheorem{definition}{Definition}[section]
\newtheorem{remark}{Remark}

\newtheorem{theorem}{Theorem}[section]

\newtheorem{lemma}[theorem]{Lemma}
\newtheorem{prop}[theorem]{Proposition}

\title{Ergodic Properties of Measures with Local Product Structure}
\author{Nawaf Alansari }
\date{}

\begin{document}

\maketitle

\begin{abstract}
    In this paper, we study ergodic properties of hyperbolic measures with local product structure. We show that all the classical results that hold in the case of SRB measure hold for these measures. In particular, we show the decomposition in countably many ergodic components, we prove the decomposition into K-components, and show that for hyperbolic measure with local product structure, The  K property implies the Bernoulli property. We also give some examples of measures where the results are applicable.
\end{abstract}
\section{Introduction}
In this paper we prove three main results, which taken together constitute what can be called Pesin theory for measures with local product structure. This theory was first developed by Pesin for smooth measures, in \cite{P} (see also \cite{BP}) and then later was extended to Sinai-Ruelle-Bowen measures by Ledrapier in \cite{L4}. 
Through out this paper, we assume that  $f : M \longrightarrow M$ is a $C^{1+\alpha}$ diffeomorphism of a Riemannian manifold $M$. We will consider an invariant hyperbolic measure $\mu$ for $f$, which admits a local product structure with respect to local unstable and stable manifolds (see definition \ref{local product structure} below), and for which we assume $h_\mu(x) > 0$.

\begin{theorem}\label{ergodic components}
Assume that $f$ preserves a hyperbolic measure with local product structure, then there exists invariant $\Lambda_1, \Lambda_2, \dots$ such that 
\begin{enumerate}
    \item $\mu(\bigcup_{n \geq 1} \Lambda_n) = 1$ and $\Lambda_i \cap \Lambda_j = \emptyset$ when $i \neq j$. 
    \item $\mu(\Lambda_n) > 0 $ for all $n\geq 1$.
    \item $(f|\Lambda_n, \mu|\Lambda_n)$ is ergodic for each $n \geq 1$.
\end{enumerate}
\end{theorem}

\begin{theorem} \label{K Property}
Assume that $f$ preserves a hyperbolic ergodic measure with local product structure $\mu$. Then the Pinsker partition for $\mu$ is finite and consists of sets $\Lambda^1, \dots, \Lambda^m$ such that $f(\Lambda^i) = \Lambda^{i+1}$ for $i = 1, \dots, m-1$ and $f(\Lambda^m) = \Lambda^1$. For each $i =1, \dots m$, we have that $(f^m|\Lambda^i, \mu|\Lambda^i)$ is a K-automorphism. 
\end{theorem}

\begin{theorem} \label{Bernoulli property}
Assume that $f$ preserves a hyperbolic measure with local product structure $\mu$. Assume furthermore that $\mu$ has the K-property. Then $(f,\mu)$ is a Bernoulli automorphism.
\end{theorem}
\begin{remark}
For measures with local product structure the assumption that $h_{\mu}(f) > 0$ is quite natural in view of lemme, whose proof can be found in the appendix. \ref{entropy zero implies periodic}.
\end{remark}
\begin{lemma}\label{entropy zero implies periodic}
If $\mu$ is a hyperbolic measure with local product structure, then $h_\mu(f) = 0$ if and only if $\mu$ is supported on periodic orbits.
\end{lemma}

\subsection{Plan of the paper}
In the first part of section 2, we start by defining the notion of product structure on a measure space, and what it means for a measure to have a product structure. We also prove some results concerning relations between measures with product structure and other product measures one can define on the space. In the second part of section 2, we introduce the notion of a hyperbolic measure with local product structure. In section 3, we define $\epsilon$-coverings by rectangles and establish their existence for the measures under consideration. The central result in this section is Lemma \ref{Lusin based theorem}. In section 4, we define $(\lambda, \mu)$-regular partitions, establish their existence, and prove that such partitions are almost $u$-saturated. In section 5 we prove Theorem \ref{ergodic components} , which describes the ergodic components of measures with local product structure. Section 6 deals with the K-property and the proof that each ergodic component of positive measure breaks into finitely many components on which the system has the K property. In section 7, we prove the Bernoulli property for measures with local product structure, assuming that they satisfy the K-property. In section 8 we give some examples where our results apply. We collect some abstract measure theoretic lemmas that are used repeatedly throughout the paper, and for which we could not find an explicit source in the literature in the appendix. 

\subsection{Acknowledgment} 
I would like to thank my advisor, Prof  Y. Pesin for introducing me to this problem and for the many useful discussions we had while I was working on this paper.   
\begin{section}{Measures with Local Product Structure}

\subsection{Product Structure}
\begin{definition}
Let $M$ be a Riemannian manifold. We say that $X \subset M$ has \textit{product structure} if there are two transversal continuous families of smooth embedded disks $\left(V_{ver}(x)\right)_{x \in X}$ and $\left(V_{hor}(x) \right)_{x \in X}$, such that
\begin{enumerate}
    \item $V_{hor}(x_1) = V_{hor}(x_2) $ or $V_{hor}(x_1) \cap V_{hor}(x_2) = \emptyset$, for all $x_1, x_2 \in X$,
    \item $V_{ver}(x_1) = V_{ver}(x_2) $ or $V_{ver}(x_1) \cap V_{ver}(x_2) = \emptyset$, for all $x_1, x_2 \in X$,
    \item $V_{hor}(x_1) \cap V_{ver}(x_2)$ consists of exactly one point, 
    \item  $X = \left(\bigcup_{x \in X} V_{hor}(x) \right) \cap \left(\bigcup_{x \in X} V_{ver}(x) \right)$.
    
\end{enumerate}

Let $\xi_{ver}$, $\xi_{hor}$ the measurable partitions $\left(V_{ver}(x) \cap X \right)_{x\in X}$ and $\left(V_{hor}(x) \cap X \right)_{x\in X}$. \end{definition}
Let $\mathcal{F}$ be the Borel $\sigma$-algebra of $X$. Then, given a measurable partition $\xi$ of $X$, we can define a sub $\sigma$-algebra $\mathcal{F}_\xi \subset \mathcal{F}$ by 
$$
\mathcal{F}_\xi := \left\{A \in \mathcal{F} : A = \bigcup_{x \in A} \xi(x) \right\}.
$$
We can also equip the elements of the partition $\xi$ with a $\sigma$-algebra $\mathcal{F}_\xi(C)$, for $C\in \xi$, defined by 
$$
\mathcal{F}_\xi(C) := \left \{A \cap C : A \in \mathcal{F} \right\}.
$$
\begin{definition}
Let $X$ be a set with product structure, and $\xi_{ver}$ and $\xi_{hor}$ be the two measurable partitions obtained from the product structure. The \textit{vertical holonomy } maps $\pi_{x, y} :\xi_{hor}(x) \longrightarrow \xi_{hor}(y)$ is given by 
$$
\pi_{x,y}(z) = \xi_{ver}(z) \cap \xi_{hor}(y).
$$
This map is well defined under our assumptions.
\end{definition}

\begin{lemma}
Let $X$ be a set with product structure, and $\xi_{ver}$ and $\xi_{hor}$ be the two measurable partitions obtained from the product structure. Then the vertical holonomy map $\pi_{x,y}$ is an isomorphism between the measure spaces $(\xi_{hor}(x), \mathcal{F}_{\xi_{hor}}(x))$ and $(\xi_{hor}(y), \mathcal{F}_{\xi_{hor}}(y))$. 
\end{lemma}
\begin{proof}
Since $\pi_{x, y}^{-1} = \pi_{y, x}$, we only need to show that $\pi_{x, y}$ is measurable. Now, the lemma follows from the continuity of the vertical holonomy map $\pi_{x,y}$, which follows from the continuity of the vertical family $\left(V_{ver}(x)\right)_{x \in X}$. 
\end{proof}

\begin{lemma}
Let $X$ be a set with product structure, and $\xi_{ver}$ and $\xi_{hor}$ be the two measurable partitions obtained from the product structure. For any $p \in X$, let us define the product set 
$$
X_p := \xi_{hor}(p) \times \xi_{ver}(p)
$$
equipped with the $\sigma$-algebra generated by $\mathcal{F}(\xi_{hor}(p)) \times \mathcal{F}(\xi_{ver}(p))$. Let $\pi_p:X_p \longrightarrow X$ be the map defined by 
$$
\pi_p(x,y) := \xi_{ver}(x) \cap \xi_{hor}(y).
$$
Then $\pi_p$ is a measurable isomorphism between $(X_p, \sigma(\mathcal{F}(\xi_{hor}(p))\times \mathcal{F}(\xi_{ver}(p))$ and $(X, \mathcal{F})$. 
\end{lemma}
\begin{proof}
We can in fact, under our assumptions, prove that $\pi_p$ is a homeomorphism. This follows from the fact the families $\left( V_{hor}(x)\right)_{x\in X}$ and $\left( V_{ver}(x)\right)_{x\in X}$ are continuous. Now, since $\mathcal{F}$ and $\sigma(\mathcal{F}(\xi_{hor}(p))\times \mathcal{F}(\xi_{ver}(p))$ are the Borel $\sigma$-algebras of $X$ and $X_p$, the claims follows. 
\end{proof}

Now, let $\mu$ be a Borel measure on $(X, \mathcal{F})$. We introduce a family of product measures $(\mu_p^{\mathcal{P}})_{p \in X}$ as follows 

\begin{definition}
Let $X$ be a set with product structure, and $\xi_{ver}$ and $\xi_{hor}$ be the two measurable partitions obtained from the product structure. Let $\mu$ be a Borel measure on $(X, \mathcal{F})$. For any $p \in X$, we define the product measure $\mu_p^{\mathcal{P}}$ on $X$ to be 
$$
\mu_p^{\mathcal{P}} := (\pi_p)_*\left( \mu_{\xi_{hor}(p)} \times \tilde \mu_{\xi_{ver}(p)} \right ),
$$
where $\mu_{\xi_{hor}(p)}$ is the conditional  measure, and $\tilde \mu_{\xi_{ver}(p)}$ is the factor measure obtained from Rokhlin's disintegration theorem. The latter measure can be realized as a measure on $\xi_{ver}(p)$.
\end{definition}
In general, the measures $\mu_p^{\mathcal{P}}$ might not have any relation to the measure $\mu$. A condition for these measures to be related to each other is for the measure $\mu$ to behave well with the product structure on the space $X$. We formalize this idea in the next definition

\begin{definition}\label{measure product strucutre}
Let $X$ be a set with product structure, and $\xi_{ver}$ and $\xi_{hor}$ be the two measurable partitions obtained from the product structure. Let $\mu$ be a Borel measure on $(X, \mathcal{F})$. We say that the measure $\mu$ has \textit{product structure} with respect to $\xi_{hor}$ and $\xi_{hor}$ if for $\mu$-a.e. $x, y\in X $, the measure $\mu_{\xi_{hor}(y)}$ and the pushforward measure $(\pi_{x,y})_*(\mu_{\xi_{hor}(x)})$ are equivalent. 
\end{definition}

\begin{lemma}\label{density lemma}
Let $X$ be a set with product structure, and $\xi_{ver}$ and $\xi_{hor}$ be the two measurable partitions obtained from the product structure. Assume furthermore that the measure $\mu$ has a product structure with respect to $\xi_{hor}$ and $\xi_{ver}$. Then, for $\mu$-a.e. $p \in X $, the measures $\mu$ and $\mu_p^{\mathcal{P}}$ are equivalent. Furthermore, we have that 
$$
\frac{d\mu}{d\mu_p^{\mathcal{P}}}(x) = \frac{d(\pi_{x, p})_*(\mu_{\xi_{hor}(x)})}{d\mu_{\xi_{hor}(p)}}(\pi_{x, p}(x)).
$$
\end{lemma}
\begin{proof}
Let $\hat \mu = \pi_p ^* (\mu$) be the pull back of $\mu$ to $X_p$. Let $\hat \xi_{ver}$ and $\hat \xi_{hor}$ be the pull backs of the partitions $\xi_{ver}$ and $\xi_{hor}$ under the map $\pi_p$. Then, it is straightforward to see that $\hat \xi_{ver}(x,y) =  \{x\} \times \xi_{ver}(p)$ and $\hat \xi_{hor}(x,y) = \xi_{hor}(p) \times \{y\}$. Now, for $y \in \xi_{ver}(p)$, let $\hat \mu_y$ be the conditional measure of $\mu$ on $\hat \xi_{hor}(p, y)$. By the product structure of the measure $\mu$, the measures $(\hat \pi_{y, y'})_* \mu_y$ and $\mu_{y'}$ are equivalent for $\tilde \mu_{\xi_{ver}(p)}$-a.e. $y, y' \in  \xi_{ver}(p)$. Applying lemma \ref{equiv of product measures}, we see that the measures $\hat \mu$ and $\mu_{\xi_{hor}(p)} \times \tilde \mu_{\xi_{ver}(p)}$ are equivalent, and satisfy
$$
\frac{d \hat \mu}{d \mu_{\xi_{hor}(p)} \times \tilde \mu_{\xi_{ver}(p)}}(x, y) = \frac{d(\hat \pi_{y, p})_*(\mu_{y})}{d\hat\mu_{p}}(\hat \pi_{y, p}(x,y)).
$$
The lemma then follows by pushing the measures $\mu_{\xi_{hor}(p)} \times \tilde \mu_{\xi_{ver}(p)}$ and $\hat \mu$ to $X$ using the map $\pi_p$.
\end{proof}

\begin{lemma}\label{composition trick}
Let $X$ be a set with product structure, and $\xi_{ver}$ and $\xi_{hor}$ be the two measurable partitions obtained from the product structure. For $p \in X$, let $\mu_{p}^{\mathcal{P}}$ be the product measure at $p$. Let us define the density $\rho_p(x) = \frac{d\mu}{d\mu_p^{\mathcal{P}}}$. Then for any point $y \in \xi_{ver}(x)$, we have 
$$
\frac{\rho_p(x)}{\rho_p(y)} = \rho_y(x).
$$
\end{lemma}

\begin{proof}
This follows from the composition property of the holonomy maps 
$$
\pi_{y, p} \circ \pi_{x,  y} = \pi_{x,p},
$$
which implies that 
$$
\frac{d(\pi_{x, p})_*(\mu_{\xi_{hor}(x)})}{d\mu_{\xi_{hor}(p)}}(\pi_{x, p}(x)) = \frac{d(\pi_{x, y})_*(\mu_{\xi_{hor}(x)})}{d\mu_{\xi_{hor}(y)}}(\pi_{x, y}(x)) \frac{d(\pi_{y, p})_*(\mu_{\xi_{hor}(y)})}{d\mu_{\xi_{hor}(p)}}(\pi_{y, p}(y)).
$$
Now, apply lemma \ref{density lemma} to get 
$$
\rho_p(x) = \frac{d(\pi_{x, p})_*(\mu_{\xi_{hor}(x)})}{d\mu_{\xi_{hor}(p)}}(\pi_{x, p}(x)).
$$
Therefore, we have $\rho_p(x) = \rho_y(x) \rho_p(y)$, from which the lemma follows.
\end{proof}

\subsection{Hyperbolic Measures With Local Product Structure}
In this subsection, we introduce a class of hyperbolic invariant measures $\mu$, for a $C^{1+\alpha}$ diffeomorphism which admit a \textit{local product stucture}. We start by stating the classic definition of a hyperbolic rectangle, then we use this definition, and the definition of product structure established in the previous section. In this section, and throughout this paper, we denote the normalized restriction of a measure $\mu$ to a measurable set $A$ by $\mu_A$.

\begin{definition} \label{Rectangle}
Let $f$ be a $C^{1+\alpha}$ diffeomorphism of a smooth compact Riemannian manifold $M$. Given an integer $j$ between $1$ and $\dim(M)-1$,  $0<\lambda<1<\nu$, $\epsilon>0$, and $\ell>0$ a number. Consider the Pesin block $\Lambda^\ell:= \Lambda^\ell_{\lambda, \nu, \epsilon, j}$ (see \cite{BP}[Definition 2.2.7] for a complete definition). We say that $R \subset \Lambda^\ell$ is a hyperbolic rectangle if for any $x, y \in R$, we have 
$$
V^s_{loc}(x) \cap V^u_{loc}(y) \in R,
$$ 
where $V^s_{loc}(z)$ and $V^u_{loc}(z)$ are the local stable and unstable manifolds for $z \in \Lambda^\ell$ (see \cite{BP}[Chapter 9.5]).   
\end{definition}

\begin{definition}\label{local product structure}
Let $f$ be a $C^{1+\alpha}$ be a diffeomorphism of a smooth compact Riemannian manifold $M$. Assume that $\mu$ is an invariant hyperbolic measure for $f$. We say that $\mu$ admits a \textit{local product structure} if for any hyperbolic rectangle $R \subset \Lambda^\ell$ of positive $\mu$-measure, one has that the measure $\mu_R$ has product structure, in the sense of definition \ref{measure product strucutre}, with respect to the measurable partitions $V^u_{R} = \left(V^u_{loc}(x)\cap R \right)_{x \in R}$ and $V^s_{R} = \left(V^s_{loc}(x)\cap R \right)_{x \in R}$. Furthermore, there is a constant $C\geq 1$, depending on $R$, such that for $\mu$-a.e. $x, y \in R$, we have that 
$$
C^{-1} < \frac{d\left((\pi_{x,y})_* \mu_{V^u_{R}(x)} \right)}{d\mu_{V^u_R(y)}} < C.
$$
\end{definition}

\begin{remark}
The local product structure without the condition 
$$
C^{-1} < \frac{d\left((\pi_{x,y})_* \mu_{V^u_{R}(x)} \right)}{d\mu_{V^u_R(y)}} < C
$$
allows us to control subsets of measure zero with respect to the conditional measures along the unstable manifold, which turns out to be enough for the proof of ergodicity and the K property. On the other hand, this condition is absolutely essential in the proof of the the Bernoulli property, and in particular, in the proof of the existence of $\epsilon$-regular partitions by rectangles. The condition allows us to control the measure of measurable subsets of small measure and their images under holonomy maps.
\end{remark}
It follows from lemma \ref{density lemma} that for $\mu_R$-a.e. point $p \in R$, we have a product measure, which we denote by $\mu_{p, R}^{\mathcal{P}}$, and that these measure are equivalent to the measure $\mu_R$, with densities satisfying
$$
\rho_{p,R}^{\mathcal{P}}(x) := \frac{d\mu_R}{d\mu_{p, R}^{\mathcal{P}}}(x) = \frac{d((\pi^s_{x,p})_* \mu_{V^u_R(x)})}{\mu_{V^u_R(p)}}(\pi_{x,p}(x)), 
$$
where $\mu_{V^u_R(x)}$ is the conditional measure on $V^u_R(x)$ from Rokhlin disintegration theorem, and $\pi^s_{x, y}: V^u_{R}(x) \longrightarrow V^u_{R}(y)$ is the stable holonomy map between the unstable manifolds of $x$ and $y$ restricted to the rectangle $R$. We also define the other density
$$
J_{p,R}^{\mathcal{P}}(x) := \frac{d\mu_{p, R}^{\mathcal{P}}}{d\mu_R}(x).
$$

\end{section}

\section{ Partitions by Rectangles}

\begin{definition} \label{epsilon rugualr covering}
An $\epsilon$-covering by rectangles of the manifold $M$ with respect to a hyperbolic measure $\mu$ is a finite collection of disjoint rectangles of positive $\mu$ measure,  $\mathcal{R}_\epsilon$ such that 
\begin{enumerate}
    \item $\mu(\bigcup_{R \in \mathcal{R}_\epsilon}R) \geq  1 - \epsilon$; 
    \item For every $R \in \mathcal{R}_\epsilon$, there is a subset $G \subset R$, such that $\mu(G) > (1-\epsilon) \mu(R)$, and a point $p \in R$, such that for any $x \in G$, we have 
    $$
    \left | \frac{d\mu_{R}}{d\mu^{\mathcal{P}}_{p,R}}(x) -1 \right | < \epsilon,
    $$
    where $\mu^{\mathcal{P}}_{p,R}$ is the  product measure at the point $p$ from the local product structure.
\end{enumerate}
\end{definition}

The main goal of this section is to prove the following proposition 

\begin{prop}\label{section four main}
Let $f : M \longrightarrow M$ be a $C^{1+\alpha}$ diffeomorphism of a Riemannian manifold $M$. Let $\mu$ be an invariant ergodic hyperbolic measure, admitting a local product structure. Then for any $\epsilon>0$, one can construct an $\epsilon$-covering by rectangles $\mathcal{R}$. Furthermore, we can chose the maximum radius of such rectangles to be as small as we want, independently from the choice of $\epsilon$. 
\end{prop}

\begin{remark}
The definition of $\epsilon$-coverings (which was first introduced in \cite{CH}) and which we have defined in the beginning of this section, is of central importance to the proof of the Bernoulli property. In particular, it allows us to bypass the possibility that the Jacobian of the stable holonomy maps might not be continuous at all, which is not the case, for example, for SRB measures. It turns out we really do not need any continuity assumption at all. The basic idea, following the proof in \cite{CH}, is that by Lusin's theorem,  a measurable function is almost continuous, so even though the Jacobian of the holonomy maps may not be continuous on a rectangle, they are continuous on compact subsets of arbitrary large $\mu$ measure. 
\end{remark}

The following lemmas are classic, and can be found for example in \cite{BP}[Lemma 9.5.5 and Lemma 9.5.6]
\begin{lemma}\label{rectangle at lebesgue}
Let $\mu$ be hyperbolic measure with local product structure. Let $\Lambda \subset M$ be a hyperbolic set for $f$, and $\ell>0$ be such that $\mu(\Lambda^\ell) > 0$. Then for any Lebesgue density point $w \in \Lambda^\ell$, for the restriction $\mu$ to $\Lambda ^\ell$, there is rectangle $R$ at $w$ so that $\mu(R) > 0$.
\end{lemma}

\begin{lemma} \label{refining rectangles}
Let $R$ and $R'$ be two rectangles in $\Lambda$. Then, there are five disjoint rectangles $R_1, \dots, R_5$, such that 
$$
R \cup R' = \bigcup_{i=1}^5 R_i.
$$
\end{lemma}

\begin{lemma}\label{sub-rectangles lemma}
Any rectangle $R$ at a point $w \in \Lambda^\ell$ of positive $\mu$ measure can be partitioned into finitely many rectangles of arbitrarily small radius.
\end{lemma}

The following lemma is the heart of this section. It allows us to partition any hyperbolic rectangle of positive $\mu$-measure into smaller rectangles, such that on most of these rectangles, the restriction of $\mu$ is almost a product measure. This lemma is used in the proof of theorem \ref{section four main}. 
\begin{lemma}\label{Lusin based theorem}
Let $\mu$ be a hyperbolic measure admitting a local product structure. Let $R$ be a hyperbolic rectangle, then for any $\epsilon>0$, there is a an $\epsilon$-covering of $R$ by sub-rectangles. Furthermore, the maximum radius of each of these sub-rectangles can be taken to be as small as we want, independently from the choice of $\epsilon$.   
\end{lemma} 

\begin{proof}
\textbf{Step 1. }Fix a point $p\in R$, such that $\mu_{p, R}^{\mathcal{P}}$ is well-defined and equivalent to the measure $\mu_R$. Now, using Lusin theorem with the measurable function $J_{p, R}^{\mathcal{P}}$, we see that for any $\epsilon_1 > 0$, we can find a compact subset $E_{\epsilon_1} \subset R$, for which we have the following two properties:
\begin{enumerate}
    \item $\mu_R(E_{\epsilon_1}) \geq 1-\epsilon_1$,
    \item the restriction of function $J_{p, R}^{\mathcal{P}}$ to $E_{\epsilon_1}$ is continuous. 
\end{enumerate}
Now, since $E_{\epsilon_1}$ is compact, we see that for any $\epsilon > 0$, we can find $\delta > 0$, for which we have 
$$
d(x_1, x_2) < \delta \Longrightarrow \left |J_{p, R}^{\mathcal{P}}(x_1) - J_{p, R}^{\mathcal{P}}(x_2) \right| < \frac{\epsilon}{C},
$$
for all $x_1, x_2 \in E_{\epsilon_1}$. Since $C^{-1} < J_{p, R}^{\mathcal{P}}(x) < C$ for $\mu_R$-a.e. $x \in R$, we see that 

$$
d(x_1, x_2) < \delta \Longrightarrow \left |\frac{J_{p, R}^{\mathcal{P}}(x_1)}{J_{p, R}^{\mathcal{P}}(x_2)} - 1 \right| < \epsilon,
$$
for all $x_1, x_2 \in E_{\epsilon_1}$. \\ 

\textbf{Step 2. } Next, we partition $R$ into a $N$ sub-rectangles $R_1, R_2, \dots, R_N$ using lemma \ref{sub-rectangles lemma}, and we do that to have 
$$
\max_{1\leq i \leq N} diam(R_i) < \delta.
$$
Now, using lemma \ref{measure of bad atoms}, we see that for $\mu_R$-$\epsilon_1^{\frac12}$-a.e. sub-rectangle $R_i$ in this partition, we have $\mu_{R_i}\left(E_{\epsilon_1} \right) \geq 1 - \epsilon_1^{\frac12}$. These are going to be the sub-rectangles we use to obtain our $\epsilon$-covering. We call such sub-rectangles \textit{good}. Before we move on to the next step in the construction, let us set $\epsilon_1 = \frac{\epsilon_2^2}{C^4}$, so any good sub-rectangle satisfies $\mu_{R_i}\left(E_{\epsilon_1} \right) \geq 1 - \frac{\epsilon_2}{C^2}$, and so that we have $\mu_R$-$\frac{\epsilon_2}{C^2}$-such good sub-rectangles. \\

\textbf{Step 3. } Now, let us a fix a \textit{good} sub-rectangle $R_i$. Since we have $\mu_{R_i}(E_{\epsilon_1}) \geq 1 - \frac{\epsilon_2}{C^2}$, we can use lemma \ref{measure of bad atoms} with the partition of $R_i$ by the unstable manifolds $\left( V^u_{R_i}(x) \right)_{x \in R_i}$. So, we see that for $(\mu_{R_i}$-$\frac{\epsilon_2^{\frac12}}{C^2})$-a.e. point $x \in R_i$, we have that $\mu_{V^u_{R_i}(x)}\left(E_{\epsilon_1} \right) \geq 1 - \frac{\epsilon_2^{\frac12}}{C^2}$. We call such an unstable manifold in $R_i$ a \textit{good} unstable manifold. Now, let us fix a point $p_i \in R_i$, whose unstable manifold is a \textit{good} unstable. Then, for any point $x \in R_i$, whose unstable manifold in $R_i$, $V_{R_i}^{u}(x)$ is \textit{good}, consider the set 
$$
\pi^s_{p_i, x}\left(V^u_{R_i}(p_i) \cap E_{\epsilon_1}^{c} \right) \subset V^u_{R_i}(x).
$$
Notice that we have 
$$
\mu_{V^u_{R_i}(x)}(\pi^s_{p_i, x}\left(V^u_{R_i}(p_i) \cap E_{\epsilon_1}^{c} \right)) < C^2 \mu_{V^u_{R_i}(p_i)}(E^c_{\epsilon_1}) < C^2 \frac{\epsilon_2^{\frac12}}{C^2} = \epsilon_2^{\frac12}.
$$
The factor $C^2$ appears in the first inequality because we have the following relation between the densities $\frac{d \left( (\pi_{x_, p_i}^s)_*\mu_{V^u_{R_i}(x)} \right)}{d\mu_{V^u_{R_i}(p_i)}}$ (the density between the conditional measures along unstables manifold \textit{restricted} to the sub-rectabgle $R_i$)  and $\frac{d \left( (\pi_{x_, p_i}^s)_*\mu_{V^u_{R}(x)} \right)}{d\mu_{V^u_{R}(p_i)}}$ (the density between the conditional measures along the the whole unstable manifold in the original rectangle $R$.), 

$$
\frac{d \left( (\pi_{x_, p_i}^s)_*\mu_{V^u_{R_i}(x)} \right)}{d\mu_{V^u_{R_i}(p_i)}} = \frac{\mu_{V^u_{R}(p_i)}\left(V^u_{R_i}(p_i)\right)}{\mu_{V^u_{R}(x)}\left(V^u_{R_i}(x)\right)}  \frac{d \left( (\pi_{x_, p_i}^s)_*\mu_{V^u_{R}(x)} \right)}{d\mu_{V^u_{R}(p_i)}}.
$$
Now, we have from our assumption that the Jacobian of the holonomy maps on $R$ satisfies 
$$
C^{-1} < \frac{d \left( (\pi_{x_, p_i}^s)_*\mu_{V^u_{R}(x)} \right)}{d\mu_{V^u_{R}(p_i)}} < C,
$$
which implies that 

$$
C^{-1} < \frac{\mu_{V^u_{R}(p_i)}\left(V^u_{R_i}(p_i)\right)}{\mu_{V^u_{R}(x)}\left(V^u_{R_i}(x)\right)}  < C,
$$
leading to the bounds 
$$
C^{-2} < \frac{d \left( (\pi_{x_, p_i}^s)_*\mu_{V^u_{R_i}(x)} \right)}{d\mu_{V^u_{R_i}(p_i)}} < C^2.
$$
Therefore, we see that for each point $x \in R_i$, whose unstable manifold $V^u_{R_i}(x)$ is \textit{good}, we have 
$$
\mu_{V^u_{R_i}(x)}(\pi_{p_i,x}\left(V^u_{R_i}(p_i) \cap E_{\epsilon_1} \right) \cap E_{\epsilon_1}) > (1-\epsilon_2^{\frac{1}{2}}) + (1 - \frac{\epsilon_2^{\frac12}}{C^2}) - 1 > 1 - 2 \epsilon_2^{\frac12} .
$$

Finally, let us define the measurable set $E_i \subset R_i \cap E_{\epsilon_1}$,  by 
$$
E_i = \bigcup_{y \in G_i}\left[V^u_{R_i}(y) \cap E_{\epsilon_1} \cap \pi_{p_i, x}\left(V_{R_i}^u(p_i) \cap E_{\epsilon_1} \right) \right],
$$
where $G_i \subset R_i$ is the set of points $y \in R_i$ whose unstable manifold $V^u_{R_i}(y)$ is \textit{good}. This set $E_i$ has the following properties:
\begin{enumerate}
    \item $\mu_{R_i}(E_i) \geq (1-2
    \epsilon_2^{\frac12})(1-\frac{\epsilon_2^{\frac{1}{2}}}{C^2})$; 
    
    \item for each $x \in E_i$, we have that both points $x$ and $\pi_{x,p}^s(x)$ are in $E_{\epsilon_1}$.
\end{enumerate}
Now, we see that the second property of $E_i$ implies that $\left | \frac{J_{p, R}^{\mathcal{P}}(x)}{J_{p, R}^{\mathcal{P}}(\pi_{x, p_i}(x))} - 1 \right | < \epsilon$. Now, we can also use lemma \ref{composition trick}, to see that we have 

$$
\left | J_{p_i, R}^{\mathcal{P}}(x) - 1 \right | < \epsilon,
$$
for all $x \in E_i$. This is not exactly what we need, what we actually need to prove is that $\left | J_{p_i, R_i}^{\mathcal{P}}(x) - 1 \ \right | $ is small on $E_i$. We show this in the next step. Before we move on, let us replace $\epsilon_2 = \frac{\epsilon_3^2}{4C^4}$ to improve the readability of  the computations we are going to make in the next step. \\ 

\textbf{Step 4.} First, we have 
$$
J_{p_i, R}^{\mathcal{P}}(x) = \frac{d((\pi_{p_i, x})_*\mu_{V^u_{R}(p_i)})}{d\mu_{V^u_{R}(x)}}(x).  
$$
Therefore, since the conditional measure $\mu_{V^u_{R_i}(x)}$ along the unstable manifold of $x$ in $R_i$ is the conditional measure $\mu_{V^u_{R}(x)}$ conditioned on the set $V^u_{R_i}(x) \subset V^u_{R}(x)$, we see that 
$$
J_{p_i, R_i}^{\mathcal{P}}(x) = \frac{\mu_{V^u_R(x)}(V^u_{R_i}(x))}{\mu_{V^u_R(p_i)}(V^u_{R_i}(p_i))} J_{p_i, R}^{\mathcal{P}}(x).
$$
Hence, we have
$$
\left | J_{p_i, R_i}^{\mathcal{P}}(x) - 1 \ \right | \leq \frac{\mu_{V^u_R(x)}(V^u_{R_i}(x))}{\mu_{V^u_R(p_i)}(V^u_{R_i}(p_i))} \left | J_{p_i, R}^{\mathcal{P}}(x) - 1 \right |
 + \left | \frac{\mu_{V^u_R(x)}(V^u_{R_i}(x))}{\mu_{V^u_R(p_i)}(V^u_{R_i}(p_i))} - 1 \right |.
 $$
 Therefore, to control $\left | J_{p_i, R_i}^{\mathcal{P}}(x) - 1 \ \right |$, we will need to have a good control over $\left | \frac{\mu_{V^u_R(x)}(V^u_{R_i}(x))}{\mu_{V^u_R(p_i)}(V^u_{R_i}(p_i))} - 1 \right |$, so we do that next. Let us take an arbitrary $x \in E_i$, then we have the following 
 
 \begin{align*}
 \mu_{V^u_R(p_i)}(V^u_{R_i}(p_i))  & = \int_{V^u_{R_i}(x)}J_{p_i, R}^{\mathcal{P}}(y) d\mu_{V^u_{R}(x)} \\ 
 & = \int_{V^u_{R_i}(x) \cap E_i}J_{p_i, R}^{\mathcal{P}} (y) d\mu_{V^u_{R}(x)} + \int_{V^u_{R_i}(x) \cap E_i^c}J_{p_i, R}^{\mathcal{P}}(y)  d\mu_{V^u_{R}(x)}.
 \end{align*}
 We first consider the first integral. We have 
 $$
 \int_{V^u_{R_i}(x) \cap E_i}J_{p_i, R}^{\mathcal{P}} (y) d\mu_{V^u_{R}(x)} = (1 \pm \epsilon) \mu_{V^u_{R}(x)}(E_i) = (1 \pm \epsilon)(1-\frac{\epsilon_3}{C^2}) \mu_{V^u_{R}(x)}(V^u_{R_i}(x)).
 $$

On the other hand,
$$
\int_{V^u_{R_i}(x) \cap E_i^c}J_{p_i, R}^{\mathcal{P}}(y)  d\mu_{V^u_{R}(x)} = C^{\pm 2}\mu_{V^u_{R}(x)}(E_i^c) = C^{\pm 2} \frac{\epsilon_3}{C^2}\mu_{V^u_{R}(x)}(V^u_{R_i}(x)).
$$
So putting this all together, we get that 

\begin{align*}
\left((1 - \epsilon)(1-\frac{\epsilon_3}{C^2}) + \frac{\epsilon_3}{C^4} \right)\mu_{V^u_{R}(x)}(V^u_{R_i}(x)) &< 
\mu_{V^u_R(p_i)}(V^u_{R_i}(p_i))\\
&< \left((1 + \epsilon)(1-\frac{\epsilon_3}{C^2}) + \epsilon_3 \right) \mu_{V^u_{R}(x)}(V^u_{R_i}(x)),
\end{align*}
which can be further reduced to 
$$
(1-\epsilon - \epsilon_3) \mu_{V^u_{R}(x)}(V^u_{R_i}(x)) <\mu_{V^u_R(p_i)}(V^u_{R_i}(p_i)) <(1+\epsilon + \epsilon_3) \mu_{V^u_{R}(x)}(V^u_{R_i}(x)).
$$
Hence, 
$$
\left | \frac{\mu_{V^u_R(x)}(V^u_{R_i}(x))}{\mu_{V^u_R(p_i)}(V^u_{R_i}(p_i))} - 1 \right | < 2(\epsilon + \epsilon_3).
$$

This implies
$$
\left | J_{p_i, R_i}^{\mathcal{P}}(x) - 1 \ \right| \leq (1+2(\epsilon + \epsilon_3))\epsilon  + 2(\epsilon + \epsilon_3),
$$
for all $x \in E_i$. Now, we see that the constants $\epsilon$ and $\epsilon_3$ can be chosen to be arbitrarily small, hence we see that we can have as much control on $\left | J_{p_i, R_i}^{\mathcal{P}}(x) - 1 \ \right|$ as we want. We also see that the total measure of \textit{good} sub-rectangles $R_i$ can be taken to be as large as we want, and that $E_i$ for each rectangle has a large measure with respect to $\mu_{R_i}$. This  finishes the proof.
\end{proof}

Now, we are ready to prove the main proposition of the section \ref{section four main}
\begin{proof}[Proof of proposition \ref{section four main}]
Fix $\epsilon>0$ small enough, and a hyperbolic set $\Lambda \subset M$ so that $\mu(\Lambda) >0$. WLOG, we assume that $\mu(\Lambda) = 1$. Choose $\ell > 0$ to be large enough so that $\mu(\Lambda^\ell) > 1-\frac\epsilon 2$. We want to find rectangle $R_1, R_2, \dots, R_k$ so that 
\begin{enumerate}
    \item $\mu(R_i) > 0$ for any $i =1 , \dots, k$;
    \item $R_i \cap R_j  = \emptyset$ for any $i \neq j$;
    \item $\mu(\Lambda \backslash \bigcup_{i=1}^k R_i) < \frac \epsilon 2 $.
\end{enumerate}
We  can construct such rectangle using lemmas \ref{rectangle at lebesgue} and \ref{refining rectangles}.
Since $\mu$-a.e point in $\Lambda^\ell$ is a Lebesgue density point, lemma \ref{rectangle at lebesgue} allows us to find at most countably many rectangles $R_1, R_2, \dots$, such that 
\begin{enumerate}
    \item $\mu(R_i) > 0$, for $i = 1, 2, \dots$ ; 
    \item $\mu(\Lambda^\ell \backslash \bigcup_{i \geq 1} R_i) = 0 $.
\end{enumerate}
Now, by the two properties above, we can find $k \geq 1$ large enough so that 
$$
\mu(\Lambda^\ell \backslash \bigcup_{i =1}^k R_i) < \frac \epsilon 2 .
$$
By lemma \ref{refining rectangles}, we can assume without lose of generality that $R_i \cap R_j = \emptyset$ for any $1 \leq i \neq j  \leq k$. All of this implies that the rectangles $R_1, \dots, R_k$ satisfy all the three properties we want. Next, for each one of the constructed rectangles, $R_1,, R_2, \dots, R_k$, we use lemma \ref{Lusin based theorem} to partition each one of the rectangles so that $\mu_{R_i}$-$\epsilon$-a.e. of these sub-rectangles satisfy the second property of $\epsilon$-coverings. Now, it is a straight forward computation to show that the total measure of such good sub-rectangles is $\geq 1-\epsilon$. This gives the partition we want. Notice that the elements of this partition can have a diameter as small as we want, independently of $\epsilon$, which follows from lemma \ref{Lusin based theorem}.

\end{proof}

\begin{section}{$(\lambda, \mu)$-regular partitions}
The goal of this section is to show that one can construct a sequence of partition $\mathcal{P}_1 < \mathcal{P}_2 < \dots $, such that each of the partitions is $(\lambda, \mu)$-regular (see definition \ref{lambda mu regular partitions definition}), and such that 

$$
\bigvee_{n \geq 1} \mathcal{P}_n = \epsilon,
$$
where $\epsilon$ is the partition by points. The results of this section are only needed in the proof of the Bernoulli property. In most cases considered in the literature, the construction of such a sequence of partitions is pretty straightforward. For example, in the case $\mu$ is an absolutely continuous measure, any sequence of finite partitions $\mathcal{P}_n$'s, with piece-wise smooth boundary and which refine to the partition by points, is $(\lambda, \mu)$-regular, for any $\lambda > 1$. The story is similar in the case of SRB measures, or more generally in the case of Equilibrium measures (see for example \cite{CH}). The basic idea in these cases is to show that $(\lambda, \mu)$-regularity follows from some precise control over the unstable, or the stable Jacobian of the map with respect to the measure at hand. The main point of this section is that if one invests some effort in choosing a good partition $\mathcal{P}$, then this property holds without any requirement on the Jacobian. Indeed, this turns out to be a measure theoretical statement that does not need any dynamics to be present. 

\begin{prop}\label{constuction of lambda mu partitions}
Let $\mu$ be any Borel probability measure on $M$, and let $\lambda > 1$. Then one can find a sequence of $(\lambda, \mu)$-regular partitions $\mathcal{P}_1 < \mathcal{P}_2 < \dots$, such that 
$$
\bigvee_{n \geq 1}\mathcal{P}_n = \epsilon.
 $$
\end{prop}
After that, we show that $(\lambda, \mu)$-regular partitions interact well with hyperbolicity. In particular, we show that a lot of elements of the partition 
$$
\bigvee_{n = N}^{N'}f^n\mathcal{P}
$$
become almost saturated by unstable manifolds (see definition \ref{almost saturation def}) . 
\begin{prop}\label{almost u-saturation property}
Let $f : M\longrightarrow M$ be a $C^{1+\alpha}$ diffeomorphism, and $\mu$ an ergodic hyperbolic invariant measure for $f$. Let also $\mathcal{P}$ be a $(\lambda, \mu)$-regular partition, and $\epsilon > 0$.  Then there is a $N_0 = N_0(\epsilon)$, and $\xi_0 = \xi_0(\epsilon)$, such that for any $N' \geq N \geq N_0$, and $\xi < \xi_0$, the partition
$$
\bigvee_{n = N}^{N'}f^n\mathcal{P}
$$
is $(\epsilon, \xi)$-almost $u$-saturated.
\end{prop}
\begin{subsection}{Construction of $(\lambda, \mu)-$regular partitions}

\begin{definition}\label{lambda mu regular partitions definition}
Let $N \subset M$ be a local submanifold of $M$ of dimension $\dim(M)-1$. We say that $N$ is $(\lambda, \mu)$-regular if 
$$
\sum_{n \geq 1} \mu\left(B(N, \lambda^{-n}) \right) < \infty.
$$
We say that an open set $U$, whose closure $\bar U$ has a piece-wise smooth boundary is $(\lambda, \mu)$-regular if 
$$
\sum_{n \geq 1} \mu\left(\partial_{\lambda^{-n}} U \right) < \infty.
$$
We say that a partition $\mathcal{P}$ of $M$ is $(\lambda, \mu)$-regular if every element of the partition has an open interior, and such that closure has piecewise-smooth boundary and is $(\lambda, \mu)$-regular.
\end{definition}
\begin{lemma}\label{refine open covering }
Let $N$ be a smooth \textbf{local} submanifold of $M$, of dimension $\dim(M)-1$. If $N$ is $(\lambda, \mu)$-regular, then for any subset $V \subset N$, open in the relative topology, one has that $V$ is also $(\lambda, \mu)$ -regular.
\end{lemma}

\begin{remark}
In the literature, the Bernoulli property has been generally proved for measures with nice holonomy maps. 
\end{remark}

\begin{lemma}\label{lambdamu boundry-set}
A precompact open set $U$, whose boundary is piece-wise smooth is $(\lambda, \mu)$-regular if and only if each smooth piece of the boundary is $(\lambda, \mu)-regular$.
\end{lemma}
\begin{proof}
This immediately follows from the fact that, if we let $\partial U = D_1 \cup \dots \cup D_m$, then we have $\partial_r U = B_r(D_1) \cup \dots B_r(D_m)$. 
\end{proof}

\begin{lemma}
Let $U$ and $V$ be precomapct and open with piece-wise smooth boundary. If $U$ and $V$ are $(\lambda, \mu)$-regular, then the connected components of the sets $U \backslash (U\cap V)$, $V \backslash (U\cap V)$ and $U\cap V$ are all $(\lambda, \mu)$-regular.
\end{lemma}
\begin{proof}
This follows from the fact that each of the connected components of the sets $U \cap V $, $U \backslash U \cap V$ and $V \backslash U \cap V$ have piece-wise smooth boundaries, and each of these boundaries can be taken so that they are subsets of one of the smooth pieces of the boundaries of $U$ and $V$. Since each one of these pieces now is $(\lambda, \mu)$-regular, we can use lemma \ref{lambdamu boundry-set} to immediately see that the lemma now follows. 
\end{proof}

\begin{lemma}\label{regularity of refining partition}
Let $U_1, \dots U_k$ be precompact open sets, each with a piecewise smooth boundary, and are $(\lambda, \mu)$-regular. Then we can find a finite $(\lambda, \mu)$-regular partition $\mathcal{P}_U$ of $\bigcup _{i = 1}^{k} U_i$.
\end{lemma}

\begin{proof}
This is an immediate application of the previous lemma.  
\end{proof}

\begin{lemma}\label{regularity of joined partitions}
Let $\mathcal{P}_1$ and $\mathcal{P}_2$ be two $(\lambda, \mu)$-regular partitions. Then $\mathcal{P}_1 \vee \mathcal{P}_2$ is $(\lambda, \mu)$-regular.
\end{lemma}
Our construction of the refining $(\lambda, \mu)$-regular partitions starts with the following lemma.
\begin{lemma}\label{regular balls lemma}
Let $\mu$ be a Borel measure and $\lambda > 1$,then there is a $\delta > 0$ such that for any $x$ in $M$, we have that the ball $B_r(x)$ is $(\lambda, \mu)$-regular, for Lebesgue almost every $r \in (0, \delta)$. 
\end{lemma}

Lemma \ref{regular balls lemma} follows from the following lemma concerning general Borel measures on intervals. 

\begin{lemma}[\cite{BP}[Lemma 9.4.2]]\label{Borel measures on the interval}
Let $\mu$ be a Borel measure on the interval $I = (0, r_0)$, where $r_0 > 0$. Then for any $\lambda > 1$, the set 
$$
\left \{r \in I : \sum_{n=1}^\infty \mu \left((r-\lambda^{-n}, r+\lambda^{-n}) \right) < \infty \right \}
$$
has Lebesgue measure $r_0$.
\end{lemma}

The basic idea of the construction now is to use lemma \ref{regular balls lemma} to construct a covering of $M$ by finitely many $(\lambda, \mu)$-regular balls of arbitrarily small radius, and then use lemma \ref{refine open covering } to construct a $(\lambda, \mu)$-regular covering. The main issue then is to ensure that the sequence of coverings we get are increasing. We will do that by using lemma \ref{regularity of joined partitions}.
\begin{proof}[Proof of proposition \ref{constuction of lambda mu partitions}]
Let $\mathcal{P}$ be a $(\lambda, \mu)$-regular partition. Fix a small $\epsilon>0$. We wish to show that there is a $(\lambda, \mu)$-regular partition $\mathcal{P}'$ such that $\mathcal{P} < \mathcal{P}' $, and the radius of each of the elements of the partition $\mathcal{P}'$ is smaller than $\epsilon$. To do this, we first use lemma \ref{regular balls lemma} to construct a finite covering of $M$ by balls of radius smaller than $\epsilon$. Then, we use lemma \ref{regularity of refining partition} to construct a partition $\mathcal{Q}$ of $M$. By construction, each element of this partition has radius smaller than $\epsilon$, and $\mathcal{Q}$ is a $(\lambda, \mu)$-regular partition. Next, we define $\mathcal{P}' := \mathcal{P} \vee  \mathcal{Q}$, then lemma \ref{regularity of joined partitions} insures that $\mathcal{P}'$ is a $(\lambda, \mu)$-regular partition, it is finite and by construction, each element of this partition has radius smaller than $\epsilon$. Now to prove the proposition, we start with any $(\lambda, \mu)$-regular partition. Then, let us take for example $\epsilon_n = \frac{1}{2^n}$. Then using the above, we can inductively construct a sequence of partitions, $\mathcal{P}_1 < \mathcal{P}_2 < \dots $, such that each partition has elements of radius less than $\epsilon_n$. This immediately implies that $\bigvee_{n\geq 1}\mathcal{P}_n = \epsilon$.

\end{proof}
\end{subsection}

\begin{subsection}{$(\lambda, \mu)$-regular partitions and almost $u$-saturation}
The following definition can be found in \cite{exponential mixing}[Definition 4.1]. The almost $u$-saturation property of $(\lambda, \mu)$-regular partitions follows from similar arguments to the proof of  \cite{exponential mixing}[lemma 4.2]. 
\begin{definition}[\cite{exponential mixing}[Definition 4.1]]\label{almost saturation def}
we say that a set $A \subset M$ is $(\epsilon, \xi)$-u-saturated if there a subset $E \subset A$, with $\mu(E) \geq (1-\epsilon)\mu(A)$, and such that for every $x \in E$, we have $V^u_\xi(x) \subset A$. We say that a finite partition $\mathcal{P}$ is$(\epsilon, \xi)$-u-saturated if $\epsilon$-a.e. atom of $\mathcal{P}$ is $(\epsilon, \xi)$-u-saturated.  
\end{definition}
\begin{proof}[Proof of proposition \ref{almost u-saturation property}]
Let us fix small $\epsilon>0$. Take $\ell>0$ so that $\mu(\Lambda^\ell) \geq 1- \epsilon^2/2$, and let us fix $\xi < r_\ell$, where $r_\ell$ is the size of unstable manifolds of points in $\Lambda^\ell$ (see \cite{BP}[(8.4)] for example). Let $ P \in \mathcal{P}$, and $x \in f^k(P) \cap \Lambda^\ell$, and assume that the $V^u_{\xi}(x)$ intersects the boundary of $P$. Then by the uniform hyperbolicity of $\Lambda^\ell$, we have $f^{-k}(x) \in \partial_{C_\ell \lambda^{-k} \xi }(P)$, where the constant $C_\ell$ depends on $\Lambda^\ell$. Therefore, the set of points $x \in  \Lambda^\ell$ for which $V^u_\xi(x)$ intersect the boundary of the partition element $f^k(P)$, for some $k \geq  N$ and for some $P \in \mathcal{P}$ is at most $\sum_{k \geq N}\mu(f^k\left(\partial_{C_\ell \lambda^{-k} \xi}(\mathcal{P})\right)) = \sum_{k \geq N}\mu(\partial_{C_\ell \lambda^{-k} \xi}(\mathcal{P}))$. Now, since $\mathcal{P}$ is $(\lambda, \mu)$-regular, we see that we can take $N = N(\epsilon)$ large enough so that we have $\sum_{k \geq N} \mu (\partial_{C_\ell \lambda^{-k}\xi}(\mathcal{P})) < \epsilon^2/2$. This means that for any $N' \geq N$, the set of points in $x \in M$ for which $V^u_\xi(x)$ is not completely contained in the atom of $\bigvee_{i=N}^{N'}f^i\mathcal{P}$ containing the point $x$ have measure at most $\epsilon^2$. The proposition is proved now by applying lemma \ref{measure of bad atoms}. 
\end{proof}
\end{subsection}
\end{section}

\begin{section}{Ergodic Components}
We start with the following well-known lemma, which we prove here for the sake of completeness. Let $\Lambda = \Lambda_{\lambda, \nu, j}$ be hyperbolic block.

\begin{lemma} \label{constant on invariants}
Let $\mu$ be a hyperbolic measure, and let $\varphi \in L^1(\mu)$ be an $f$-invariant function. Then there is a set $N_\varphi \subset M$ such that 
\begin{enumerate}
    \item $\mu(N_\varphi) = 0$,
    \item For any $w \in \Lambda$, any $x, y\in V^s(w) \backslash N_\varphi$, we have $\varphi(x) = \varphi(y)$,
    \item For any $w \in \Lambda$, any $x, y\in V^u(w) \backslash N_\varphi$, we have $\varphi(x) = \varphi(y)$.
\end{enumerate}

\end{lemma}

\begin{proof}
Let $\phi:M \longrightarrow \mathbb{R}$ be a continuous function. Define 
\begin{gather*}
\phi^+(x) := \lim _{n \rightarrow \infty}\frac1n \sum^{n-1}_{k=0} \phi(f^k(x)), \\
\phi^-(x) := \lim _{n \rightarrow \infty}\frac1n \sum^{n-1}_{k=0} \phi(f^{-k}(x)), \\
\bar\phi(x) := \lim _{n \rightarrow \infty}\frac1{2n-1} \sum^{n-1}_{k=-(n-1)} \phi(f^k(x)). 
\end{gather*}
By Birkhoff Ergodic Theorem, there is a set $N_\phi \subset M$ such that $\mu(N_\phi)= 0$, and for any $x \in M \backslash N_\phi$, $\phi^+(x)$, the function $\phi^-(x)$ and $\bar \phi(x)$ exist and are equal to each other. Now, notice that because of the continuity of $\phi$ and the fact that for any $x, y \in V^s(w) \backslash N_\phi$, we have $d(f^n(x), f^n(y)) \longrightarrow 0$ as $n \rightarrow \infty$ 
$$
\bar \phi(x) = \phi^+(x) = \phi^+(y) = \bar \phi(w).
$$
Now, by the density of continuous function in $L^1(\mu)$, we can find a sequence $(\phi_k)_k$ of continuous functions such that $\| \varphi - \phi_k \|_{L^1(\mu)} < \frac1k$. By the invariance of $\mu$,and  $\varphi$, we have 
$$
\| \varphi - \phi_k  \circ f^i\|_{L^1(\mu)} = \| \varphi \circ f^i - \phi_k  \circ f^i\|_{L^1(\mu)}  = \| \varphi - \phi_k \|_{L^1(\mu)}< \frac1k
$$
and hence, we have 
$$
\left \| \varphi - \frac1n \sum_{i=0}^{n-1}\phi_k  \circ f^i \right\|_{L^1(\mu)} < \frac1k.
$$
Applying Lebesgue convergence theorem, we obtain 
$$
\| \varphi - \phi^+_k \|_{L^1(\mu)} < \frac1k.
$$

Hence, we can find a subsequence $\phi_{n_k}$ such that $\varphi(x) = \lim \limits_{k \rightarrow \infty} \phi_{n_k}(x)$ for $\mu$-a.e. $x \in M$. Setting $N_{\varphi} =  \bigcup_{k\geq 1} N_{\phi_{n_k}} $, we see that this set satisfies all the conditions of the lemma. 
\end{proof}

Now, Let $\mu$ be a hyperbolic measure with local product structure, $\Lambda$ be a hyperbolic set for $f$ such that $\mu(\Lambda) > 0$, and choose $\ell > 0$ large enough so that $\mu(\Lambda^\ell)> 0$. Let $w \in \Lambda^\ell$ be a Lebesgue density point for the restriction of $\mu$ to $\Lambda^\ell$, and let $R$ be a rectangle at $w$ such that $\mu(R)>0$, we define 
$$
Q:= \bigcup_{n \in \mathbb{Z}} f^n(R).  
$$

\begin{lemma}\label{ergodicity of Q}
When $R$ and $Q$ are chosen as above, then $(f|_Q, \mu_Q)$ is ergodic. Here $\mu_Q$ is the normalized restriction of $\mu$ to $Q$.
\end{lemma}
\begin{proof}
\textbf{Step 1.} Let $\varphi \in L^1(\mu)$ be an $f$-invariant, and $N_\varphi$ the set from lemma  \ref{constant on invariants}. Fix a point $p \in R$, for which we have 
\begin{enumerate}
    \item $\mu_{V^u_R(p)}\left(V^u_R(p) \cap N_\varphi \right) = 0$,
    
    \item The product measure $\mu^{\mathcal{P}}_{p, R}$ is equivalent to the measure $\mu_R$.
\end{enumerate}
We can choose such a point because both properties are satisfied for $\mu_R$-a.e. point in $R$. Next, define the set 
$$
N_{\varphi, R}^s := \bigcup_{x \in V^u_R(p) \cap N_\varphi}V^s_R(x).
$$
Then, since $\mu^{\mathcal{P}}_{p, R}\left(N^s_{\varphi, R} \right) = 0$, we also have $\mu_{R}\left(N^s_{\varphi, R} \right) = 0$.  Now, take two points $z_1, z_2 \in R \backslash \left(N_{\varphi} \cup N^S_{\varphi, R} \right)$, and let $y_i := \pi^s_{z_i, p}(z_i)$. Then, by the choice of $z_1$ and $z_2$, we have $y_1, y_2 \in V^u_{R}(p) \backslash N_\varphi$. Applying lemma \ref{constant on invariants}, we get 
$$
\varphi(z_1) = \varphi(y_1) = \varphi(y_2) = \varphi(z_2).
$$

\textbf{Step 2.} Let us define the invariant sets 
$$
N_{\varphi, \mathbb{Z}} := \bigcup_{n \in \mathbb{Z}} f^n N_{\varphi} \mbox{ and } N^s_{\varphi, \mathbb{Z}}:= \bigcup_{n \in \mathbb{Z}} f^n N^s_{\varphi, R}. 
$$
Then, since $\mu\left(N_{\varphi} \right)=0$ and  $\mu \left( N^s_{\varphi, R}\right) = 0$, we get 

$$
\mu \left(N_{\varphi, \mathbb{Z}} \right) = 0 \mbox{ and } \mu\left( N^S_{\varphi, \mathbb{Z}}\right) = 0.
$$ 
Now, let us take $z_1, z_2 \in Q \backslash \left( N_{\varphi, \mathbb{Z} } \cup N^s_{\varphi, \mathbb{Z}} \right)$, then there are $n_1, n_2 \in \mathbb{Z}$ and $z_1', z_2' \in R \backslash \left(N_{\varphi} \cup N^s_{\varphi, R} \right)$ such that $z_1 = f^{n_1}(z_1')$ and $z_2 = f^{n_2}(z_2')$. By step 1 and the invariance of $\varphi$, we get 
$$
\varphi(z_1) = \varphi(z_1') = \varphi(z_2') = \varphi(z_2).
$$
Therefore, the function $\varphi$ is constant $\mu$-a.e. on $Q$, which implies that $(f|_Q, \mu_Q)$ is ergodic. 
\end{proof}
To complete the proof of Theorem  \ref{ergodic components}, we first notice that for any two rectangle $R_1$ and $R_2$ of positive $\mu$ measure, the associated sets $Q_1$ and $Q_2$ are either identical up to a set of zero $\mu$-measure, or we have $\mu(Q_1 \cap Q_2) = 0$. This follows from the invariance of $Q_1$ and $Q_2$ and from the lemma \ref{ergodicity of Q}. Now when $\mu(\Lambda^\ell) > 0$, we know that $\mu$-a.e. $w \in \Lambda^\ell$ is a Lebesgue density point. Hence by lemma \ref{rectangle at lebesgue} and the fact that $\Lambda = \bigcup_{\ell >1} \Lambda^\ell$, we can find at most countably many points $w_1, w_2, \dots \in \Lambda$ and rectangles $R_1, R_2, \dots$ at these points, such that the each of the associated sets $Q_1, Q_2, \dots $ has positive $\mu$ measure, and such that 
\begin{enumerate}
    \item $\mu(\Lambda \backslash \bigcup_{n \geq 1} Q_n) = 0$.
    \item $\mu(Q_i \cap Q_j) = 0 $, whenever $i \neq j$.
    \item $(f|_{Q_n}, \mu_{Q_n})$ is ergodic.
\end{enumerate}
To finish the proof of theorem \ref{ergodic components}, we define 
$$
\Lambda_n := Q_n \backslash \bigcup_{k=1}^{n-1} Q_k.
$$
Then, $\Lambda_1, \Lambda_2, \dots$ form a partition of $\Lambda$, which satisfies all the conditions required in the theorem.

\end{section}

\begin{section}{Pinsker Partition and K-Property}
In this section, we will prove theorem \ref{K Property}. Let $\xi^s$ and $\xi^u$ be two partitions which are subordinate to the stable and unstable manifolds respectively. Such partitions can be constructed for any measure $\mu$, with at least one non-zero Lyapunov exponent. Let $\pi(f,\mu)$ be the Pinsker partition, this is the partition whose associated $\sigma$-algebra is the maximal $\sigma$-algebra of sets $A$, such that the partitions $\left\{A, m\backslash A \right\}$ have zero entropy (see \cite{P} and \cite{BP}[chapter 9.4] for a complete definition). Let $W^+$ and $W^-$ be the partition into global unstable and stable manifold. We denote by $\mathcal{H}(\xi)$, the measurable hull of the partition $\xi$. The following theorem can found in \cite{BP}[Theorem 9.4.1] 

\begin{prop}
Let $f:M \longrightarrow M$ be a $C^{1+\alpha}$ diffeomorphism, and $\mu$ an invariant measure with at least one non-zero Lyapunov exponent, then there is a partition $\xi^s$ subordinate to the stable manifolds, which satisfies the following properties: 
\begin{enumerate} \label{subordinate partitions}
    \item $\xi^s(x)$ is an open subset of $W^s(x)$, for $\mu$-a.e. $x \in M$.
    \item $\bigvee_{n \geq 0} f^n \xi^s  = \epsilon$ is the partition by points.
    \item $f \xi^s \geq \xi^s$. 
    \item $\bigwedge_{n \geq 0} f^{-n} \xi^s = W^-$, the partition by global stable manifolds.
\end{enumerate}
Furthermore, we have that $\pi(f, \mu) = \mathcal{H}(W^-)$. Similar statements hold for unstable manifolds. 
\end{prop}

To prove theorem \ref{K Property},  we will show that for a hyperbolic measure with local product structure, which is ergodic, the Pinsker partition $\pi(f, \mu)$ actually has an atom of positive measure. Since the Pinsker partition is invariant under $f$, this implies that  $\pi(f, \mu)$ has finitely many atoms of positive $\mu$ measure, whose union has full $\mu$ measure. This concludes the proof of the theorem, since these atoms are permuted by $f$, so replacing $f$ by $f^p$, where $p$ is the order of the generated permutation on the atoms of $\pi(f, \mu)$, we see that the restriction of $f$ to any such atom is now a $K$-automorphism, given that $h_u(f) > 0$.  \\ 

Before we start the proof, let us start with some definitions.

\begin{definition}\label{equivelence of partitions}
Let $(X, \mathcal{F}, \mu)$ be a Lebesgue probability space. Given any two partitions $\xi$ and $\eta$ of $X$ , we say that $\xi \leq \eta$ if each partition element of $\eta$ is contained in a partition element of $\xi$. We say that $\xi \leq \eta \mod \mu$ if there is a measurable subset $X_0 \subset X$ such that $\mu(X_0) = 1$, and such that $\xi | X_0 \leq \eta |X_0 $. In particular, we say that $\xi = \eta \mod \mu$ if $\xi \leq \eta \mod \mu$ and $\eta \leq \xi \mod \mu$. 
\end{definition}

From now on, given a $\sigma$-subalgebra $\mathcal{F'} \subset \mathcal{F}$, we denote by $\Xi \left(\mathcal{F'} \right )$ the partition generated by $\mathcal{F'}$. This partition can be constructed as follows:  the probability space $(X, \mathcal{F}', \mu)$ is separable, a fact that follows from the separability of $(X, \mathcal{F}, \mu)$. Now, take a countable sequence of sets $\left(A_n \right)_{n\geq 1}$ in $\mathcal{F}'$, which is dense in $\mathcal{F}'$ with respect to the metric $d_\mu \left(A, B \right) = \mu \left(A \Delta B \right)$. We define the sequence of finite partitions 
$$
\xi_n := \bigvee_{k=1}^n \left \{A_k, X \backslash A_k \right\}.
$$
Notice that by definition, we have that $\xi_1 \leq \xi_2 \leq \dots$. we define $\Xi \left (\mathcal{F}' \right)$ by 
$$
\Xi \left(\mathcal{F}' \right) = \bigvee_{n=1}^\infty \xi_n.
$$
The fact that $\mathcal{F}\left(\Xi \left( \mathcal{F}'\right) \right)$ is equivalent to the $\sigma$-algebra $\mathcal{F}'$ follows from the fact that the union of the $\sigma$-algebra $\bigcup_{n \geq 1} \mathcal{F}\left( \xi_n\right)$ generates $\mathcal{F}'$.  

\begin{definition}
Let $\xi$ be a partition of a Lebesgue probability space $(X,\mathcal{F}, \mu)$. We denote by $L^2_\mu \left(X, \xi \right)$ the subspace of $L^2_\mu(X)$ of functions measurable with respect to the the completion of the $\sigma$-algebra $\Xi \left( \xi \right)$. 
\end{definition}

\begin{lemma}[\cite{Ro2}]
Let $\xi$ and $\eta$ be two measurable partitions be a Lebesgue probability space $(X, \mathcal{F}, \mu)$. Then the following are equivalent 
\begin{enumerate}
    \item $\xi = \eta \mod \mu$. 
    
    \item The completions of the $\sigma$-subalgebras $\Xi\left( \xi \right)$ and $\Xi \left(\eta \right)$ are equal. 
    
    \item $L^2_\mu \left(X, \xi \right) = L^2_\mu \left(X, \eta \right)$.
\end{enumerate}
\end{lemma}

\begin{definition}
For a probability space $(X, \mathcal{F}, \mu)$, we say that measurable subset $A \in \mathcal{F}$ is an \textit{Atom} if A satisfies 
\begin{enumerate}
    \item $\mu(A) > 0$.
    \item if $A' \subset A$ is a measurable set in $\mathcal{F}$, then we either have $\mu(A \Delta A') = 0$ or $A'$ has zero $\mu$ measure. 
\end{enumerate}
\end{definition}

Given a rectangle $R$ at a point $w \in \Lambda^\ell$, such that $\mu(R) > 0$, let us denote the restriction of partitions to this rectangle, by 
\begin{itemize}
    \item $\xi^u_R := \xi^u|_R$,
    \item $\xi^s_R := \xi^s|_R$,
    \item $\pi_R := \pi(f, \mu)|_R$.
\end{itemize}
Observe that $\xi^u_R$ and $\xi^s_R$ refine the partitions $\{V^u_R(x)\}_{x \in R}$ and $\{V^s_R(x)\}_{x \in R}$ of $R$ by local unstable and stable manifolds. To be precise, for $\mu$-a.e. $x \in R$, we have that for $\mu_{V^u(x)}$-a.e. $y \in V^u(x)$, $\xi^u_R(y) \subset V^u(x)$ is an open subset, and similarly for the stable partitions. \\

Our starting point is the fact that $\pi_R \leq \xi^s_R$ and $\pi_R \leq \xi^u_R$ $mod \ \ \mu$, which is a consequence of proposition \ref{subordinate partitions}. This implies that $L^2_\mu(R, \pi_R) \subset L^2_\mu(R,\xi^s_R)$ and $L^2_\mu(R, \pi_R) \subset L^2_\mu(R, \xi^u_R)$, which implies that 

$$
L^2_\mu(R, \pi_R) \subset L^2_\mu(R, \xi^u_R) \cap L^2_\mu(R, \xi^s_R).
$$
our goal is to show that we can find a rectangle $R$ for which we have $\pi_R = \{\emptyset, R\} \mod \mu $.  To show this, it is enough to find a rectangle $R$ for which we have $L^2_\mu(R, \xi^u_R) \cap L^2_\mu(R, \xi^s_R)$ is trivial, that is it consists of $L^2$ functions which are $\mu$-a.e. constant on $R$.  The first step is to observe that from the construction of stable and unstable subordinate partitions $\xi^s$ and $\xi^u$, one can find two $\mu$-measurable functions $r^s, r^u : M \longrightarrow (0, \infty)$, which satisfy 
$$
r^u(x) \leq d(x, \partial \xi^u(X)) 
$$ 
and 
$$
r^s(x) \leq d(x, \partial \xi^s(x)).
$$

Now, for any small $\delta>0$, let us define the measurable sets 
$$
R^u_\delta := \{x \in M \ | \ r^u(x) \geq \delta \},
$$
$$
R^s_\delta := \{x \in M \ | \ r^s(x) \geq \delta \}.
$$
Then, since $\mu(R^u_\delta) \longrightarrow 1 $ as $ \delta \rightarrow 0$, and similarly for $R^s_\delta$, one can find for any $\epsilon>0$ a number $\delta>0$ for which $\mu(R^u_\delta) \geq 1 - \frac{1}{3}\epsilon$ and $\mu(R^s_\delta) \geq 1 - \frac{1}{3}\epsilon$. By taking $\ell>1$ to be large enough, we can also ensure that $\mu(\Lambda^\ell) \geq  1-\frac{1}{3}\epsilon$. Putting all this together, we find that $\mu(\Lambda^\ell \cap R^u_\delta \cap R^s_\delta) > 1 - \epsilon$. Now, let us take a rectangle $R'$ in $\Lambda^\ell$ of radius $\delta^2$, such that we have $\mu(R' \cap R^u_\delta \cap R^s_\delta) > 0 $. Let us now define the rectangle $R \subset R'$ by 
$$
R := \left( \bigcup_{x \in R' \cap R^u_\delta} \xi^u(x)  \right) \bigcap \left( \bigcup_{x \in R' \cap R^s_\delta} \xi^s(x)  \right).
$$
Notice that $R$ is indeed a rectangle, since for $x \in R' \cap R^u_\delta$, we have $ \xi^u_R(x) = V^u_R(x)$ and for each $x \in R' \cap R^s_\delta$ we have $\xi^s_R(x) = V^s_R(x)$. one can easily see that $\mu(R) > 0$. Now if we take $\varphi \in L^2_\mu(R, \xi^u_R) \cap L^2_\mu(R, \xi^s_R)$, all we need to show is that for $\mu$-a.e. $x, y \in R$, we have $\varphi(x) = \varphi(y)$. Define the following sets 
$$
G^u := \left\{ x \in R : \varphi(y) = \varphi(x) \mbox{, for } \mu_{V^u_R(x)}\mbox{-a.e. } y \in V^u_R(x)   \right\},
$$

$$
G^s := \left\{ x \in R : \varphi(y) = \varphi(x) \mbox{, for } \mu_{V^s_R(x)}\mbox{-a.e. } y \in V^s_R(x)   \right\}.
$$
By our choice of $\varphi$, these two sets have full $\mu$-measure in $R$. Define the set 
$$
G = \left \{x \in G^u : \mu_{V^u_R(x)} \left(V^u_R(x) \cap G^s\right) = 1 \right \}.
$$
We see that $G$ also has full $\mu$-measure in $R$. Now, we show that $\varphi(x_1) = \varphi(x_2)$ for $\mu$-a.e. $x_1, x_2 \in G$. By the local product structure
$$
\mu_{V^u_R(x_2)} \left(\pi^s_{x_1, x_2}\left(V^u_R(x_1) \cap G^s \right) \cap \left(V^u_R(x_2) \cap G^s \right) \right) = 1 .
$$
Now, the following holds:
\begin{enumerate}
    \item $\varphi(x_1) = \varphi(y_1)$, for $\mu_{V^u_R(x_1}$-a.e. $y_1 \in V^u_R(x_1) \cap G^s$;
    
    \item $\varphi(x_2) = \varphi(y_2)$, for $\mu_{V^u_R(x_1)}$-a.e. $y_2 \in \pi^s_{x_1, x_2}\left(V^u_R(x_1) \cap G^s \right) \cap \left(V^u_R(x_2) \cap G^s \right)$;
    
    \item $\varphi(y_2) = \varphi(\pi^s_{x_2,x_1}(y_2))$, for $\mu_{V^u_R(x_1)}$-a.e. $y_2 \in \pi^s_{x_1, x_2}\left(V^u_R(x_1) \cap G^s \right) \cap \left(V^u_R(x_2) \cap G^s \right)$.
\end{enumerate}
Therefore, 
$$
\varphi(x_1) = \varphi(y_1) = \varphi(y_2) = \varphi(x_2).
$$
\end{section}

\begin{section}{Bernoulli Property}
In the section, we assume that $\mu$ is hyperbolic measure with local product structure and such that $(f, \mu)$ has the K-property. Let $\mathcal{P}$ be a $(\lambda, \mu)$-regular partition ($\lambda$ will be determined later), the main focus of this chapter is to show that the partition $\mathcal{P}$ has the \textit{Very Weak Bernoulli Property}. 
\subsection{Definitions and Results}

\begin{definition}
Let $(X, \mathcal{F}_X, \mu)$ and $(Y,\mathcal{F}_Y,\nu) $ be two Lebesgue spaces. Let $P_i = \left\{ P_1^{(i)}, \dots  P_m^{(i)}\right\}$ and $Q_i = \left\{ Q_1^{(i)}, \dots  Q_m^{(i)}\right\}$, $ 1 \leq i \leq n$, be two sequences of finite partitions of $X$ and $Y$ respectively. We say that $\left( P_i\right)_i \sim \left(Q_i \right)$ if $\mu\left(P_j^{i} \right) = \nu \left(Q_j^{i} \right)$, for all $1 \leq i \leq n$ and $1 \leq j \leq m$. We say that 
$$
\bar d \left(\left( P_i\right)_{i=1}^n, \left( Q_i\right)_{i=1}^n \right) < \epsilon
$$
if there is a Lebesgue space $(Z, \mathcal{F}_Z, \lambda)$, and two sequences of finite partitions $\left( \tilde P_i\right)_{i=1}^n$ and $\left( \tilde Q_i\right)_{i=1}^n$ of $Z$, such that 
$$
\left( P_i\right)_{i=1}^n \sim \left( \tilde P_i\right)_{i=1}^n \mbox{ and } \left( Q_i\right)_{i=1}^n \sim \left(Q_i\right)_{i=1}^n 
$$
and such that 
$$
\frac{1}{n}\sum_{i=1}^n \sum_{j=1}^m \lambda \left( \tilde P_j^{(i)} \Delta \tilde Q_j^{(i)}\right) < \epsilon.
$$
\end{definition}

\begin{definition}[The Very Weak Bernoulli Property] Let $f :M \longrightarrow M$ be an invertible measurable map, which preserves a measure $\mu$, and $\mathcal{P}$ a measurable partition of $M$. We say that $\mathcal{P}$ has the \textit{Very Weak Bernoulli}(or VWB for short) property if the following holds: For every $\epsilon > 0$, there is an $N>0$ such that, for every $n>0$ and $N_2 \geq N_1 > N$, and for $\mu$-$\epsilon$-a.e. $A \in \bigvee_{k=N_1}^{N_2}f^{k}\mathcal{P}$, we have that
$$
\bar{d}\left(\left(f^{-i}\mathcal{P}\right)_{i=0}^{n-1},\left(f^{-i}\mathcal{P}|A\right)_{i=0}^{n-1}  \right) < \epsilon.
$$

\end{definition}

One can characterise the Bernoulli property using the VWB Property, as the following two theorems found in \cite{OW} show.

\begin{theorem}\label{vwb to bernoulli}
If a partition $\mathcal{P}$ of $M$ is VWB with respect to a measure $\mu$, then the system $(M, \bigvee_{k=-\infty}^{\infty}f^k\mathcal{P}, \mu)$ is Bernoulli
\end{theorem}

\begin{theorem}\label{bernoulli to bernoulli}
Let $\mathcal{P}_1 \leq \mathcal{P}_2 \leq \dots$ be a sequence of measurable partitions, such that $\bigvee_{i=1}^{\infty}\bigvee_{k=-\infty}^{\infty}f^k\mathcal{P}_i$ is a generating partition, and such that $(M, \bigvee_{k=-\infty}^{\infty}f^k\mathcal{P}, \mu)$ is a Bernoulli system, then $(f, M, \mu)$ is Bernoulli. 
\end{theorem}

We will need the following lemma (see \cite{CH}). We use this lemma to establish upper bounds on the $\bar{d}$-distance between two sequence of partitions, in terms of $\epsilon$-preserving maps, which also $\epsilon$-preserve the sequence of partitions at hand.

\begin{lemma}\label{dbar distance estimate}
Let $(X, \mu)$ and $(Y, \nu)$ be two non-atomic Lebesgue probability spaces. Let $\left(\alpha_i \right)$ and $\left(\beta_i \right)$ be two sequences of measurable partitions of $X$ and $Y$ respectively. Assume that we have a map $\psi :X \longrightarrow Y$ for which
\begin{enumerate}
    \item $\psi$ is $\epsilon$-measure preserving with respect to $\mu$ and $\nu$;
    
    \item there is a measurable $E \subset X$, such that $\mu(E) < \epsilon$, and for which we have 
    $$
    h(x, \psi(x)) < \epsilon,
    $$
    for any $x \in X \backslash E$, where 
    $$
    h(x, y) = \frac{1}{n} \sum_{\alpha_i(x) \neq \beta_i(y)} 1.
    $$
\end{enumerate}

Then $\bar{d}((\alpha_i), (\beta_i)) < C \epsilon$, where $C>0$ is a universal constant. 
\end{lemma}

In the next subsection, we will establish the VWB Property for sequences of $(\lambda, \mu)$-regular partitions. To do that, we will use the following characterization of the VWB Property, which holds for $(\lambda, \mu)$-regular partitions 

\begin{lemma} \label{vwb charactarization for lambda mu}
Let $f: M \longrightarrow M$ be a measure preserving. Let $\mathcal{P}$ be a $(\lambda, \mu)$-regular partition. Assume that for any $\epsilon>0$, there exists $N > 0$, such that for any $N_2 \geq N_1 > N$, and for $\mu$-$\epsilon$-a.e $A \in \bigvee_{N_1}^{N_2}f^k\mathcal{P}$, there is an $\epsilon$-measure preserving map 
$$
\theta: A \longrightarrow M
$$
with respect to the measures $\mu_A$ and $\mu$, such that for $\mu_A$-$\epsilon$-a.e. $x \in A$, we have $d(f^n(x), f^n(\theta(x)) < \lambda^{-n} \epsilon$, for all $n \geq 0$. Then $\mathcal{P}$ has the VWB Property. 
\end{lemma}

\begin{proof}
Fix $\epsilon > 0$, our goal is to find $N > 0$, such that for any \textit{fixed} $N_2 \geq N_1 > N$ and $n \geq 0$, we have that $\epsilon$-a.e atom $A \in \bigvee^{N_2}_{k=N_1}f^k \mathcal{P}$ satisfy 
$$
\bar{d}\left(\left(f^{-i}\mathcal{P}\right)_{i=0}^{n-1},\left(f^{-i}\mathcal{P}|A\right)_{i=0}^{n-1}  \right) < \epsilon.
$$
To do this, we first fix $\epsilon' < \frac{\epsilon}{2}$, for which the following holds: 
$$
\sum_{i=0}^{n-1} \mu \left(\partial_{\lambda^{-i}\epsilon'}\mathcal{P} \right) < \frac{\epsilon^2}{4}.
$$

We can do this in the following way. First, since $\mathcal{P}$ is a $(\lambda, \mu)$-regular partition, we can find a $k_0$ such that
$$
\sum_{k \geq k_0} \mu\left(\partial_{\lambda^{-k}}\mathcal{P} \right) < \frac{\epsilon^2}{4},
$$
so, it is enough to choose $\epsilon' < \min\left(\frac{\epsilon}{2}, \lambda^{-k_0}\right)$.  \\ 

Now, using the assumption of this lemma, we can find $N>0$ such that for any $N_2 \geq N_1 > N$, we have that for $\epsilon'$-a.e atom $A \in \bigvee^{N_2}_{k=N_1}f^k \mathcal{P}$, there is a map $\theta: A \longrightarrow M$, such that
\begin{enumerate}
    \item $\theta$ is $\epsilon'$-measure preserving w.r.t to the measures $\mu_A$ and $\mu$; 
    
    \item $d(f^n(x), f^n(\theta(x)) < \lambda ^{-n}\epsilon$.
\end{enumerate}
We will show that this $N > 0$ is the one we want for the VWB property. To do that, we first need to estimate the measure of the set 
$$
\bigcup_{i=0}^{n-1}f^{-i}(\partial_{\lambda^{-i}\epsilon'}\mathcal{P}).
$$
To do this, we notice that 
$$
\mu \left(\bigcup_{i=0}^{n-1}f^{-i}(\partial_{\lambda^{-i}\epsilon'}\mathcal{P}) \right) \leq \sum_{i=0}^{n-1}\mu(f^{-i}\partial_{\lambda^{-i}\epsilon'}\mathcal{P}) = \sum_{i=0}^{n-1}\mu(\partial_{\lambda^{-i}\epsilon'}\mathcal{P}),
$$
where we get the last equality by the measure preservence. Now, by our choice of $\epsilon'$, we see that
$$
\mu\left(\bigcup_{i=0}^{n-1}f^{-i}(\partial_{\lambda^{-i}\epsilon'}\mathcal{P})\right) < \frac{\epsilon^2}{4}.
$$

Then, using lemma \ref{measure of bad atoms}, we see that for $\frac{\epsilon}{2}$-a.e. atom $A \in \bigvee^{N_2}_{k=N_1}f^k\mathcal{P}$, 
$$
\mu_A \left(\bigcup_{i=0}^{n-1}f^{-i}(\partial_{\lambda^{-i}\epsilon'}\mathcal{P}) \right) < \frac{\epsilon}{2}.
$$

Therefore, since $\epsilon'$ is chosen so that $\epsilon' < \frac{\epsilon}{2}$, for $\epsilon$-a.e. atom $A \in \bigvee^{N_2}_{k=N_1}f^k\mathcal{P}$, we can find a map $\theta:A \longrightarrow M$ satisfying the two conditions above, and at the same time, we have 
$$
\mu_A \left(\bigcup_{i=0}^{n-1}f^{-i}(\partial_{\lambda^{-i}\epsilon'}\mathcal{P}) \right) < \frac{\epsilon}{2}.
$$

Fix such an atom $A$, and define 
$$
E_A := \left \{ x \in A \ | \ d(f^ix, f^i\theta(x)) < \epsilon' \mbox{ for } i=0, \dots \mbox{ and } x \notin \bigcup_{i=0}^{n-1}f^{-i}(\partial_{\lambda^{-i}\epsilon'}\mathcal{P}) \right \}.
$$

First, notice that $\mu_A(E_A) < \frac{\epsilon}{2} + \epsilon' < \epsilon$. Now, for any $x \in E_A$, both $f^i(x)$ and $f^i(\theta(x))$ lie in the same atom of the partition $\mathcal{P}$, for $i = 0, \dots, n-1$. This holds because $d(f^i(x), f^i(\theta(x)) < \lambda^{-i}\epsilon'$, and because $f^i(x) \notin \partial_{\lambda^{-i}\epsilon'}\mathcal{P}$. Therefore, we see that $h(x, \theta(x)) = 0$, where $h$ is defined as in lemma \ref{dbar distance estimate}. Hence, $\theta: A \longrightarrow M$ satisfy the conditions of lemma \ref{dbar distance estimate}. Therefore, 
$$
\bar{d}\left(\left(f^{-i}\mathcal{P}\right)_{i=0}^{n-1},\left(f^{-i}\mathcal{P}|A\right)_{i=0}^{n-1}  \right) < \epsilon
$$
and also this holds for $\epsilon$-a.e. atom in $\bigvee^{N_2}_{k=N_1}f^k \mathcal{P}$, which implies that $\mathcal{P}$ has the VWB Property.
\end{proof}
\subsection{The Proof}
In this section, we will show that if  $f:(M, \mu) \longrightarrow (M, \mu)$ has the K-property, and $\mu$ admits a local product structure, then the system is in fact Bernoulli. We will do this by taking a sequence of $(\lambda, \mu)$-partitions, $\mathcal{P}_1 \leq \mathcal{P}_2 \leq \dots$, such that $\bigvee_{i=1}^\infty \mathcal{P}_i = \varepsilon$, the partition into points. this sequence of partitions can be constructed as explained in Section \ref{constuction of lambda mu partitions}. Then we show that each one of these partitions has the VWB property. We then obtain the Bernoulli property for the system $(f, \mu)$ using propositions \ref{vwb to bernoulli} and \ref{bernoulli to bernoulli} stated in the last section. Therefore, it is enough to show the following 
\begin{prop} \label{lambda mu regular has vwb}
Let $f: M \longrightarrow M$ be a $C^{1+\alpha}$ diffeomorphism. Assume that $f$ preserves a hyperbolic measure $\mu$ admitting a local product structure. Assume furthermore that $(f, \mu)$ has the K-property. Let $\mathcal{P}$ be any $(\lambda, \mu)$-regular partition for $M$, then $\mathcal{P}$ has the VWB Property.
\end{prop}

We will spend the rest of this section proving proposition \ref{lambda mu regular has vwb}. We will prove this proposition using the characterization of the VWB property we have established in the last subsection in lemma \ref{vwb charactarization for lambda mu}. \\ 

\begin{proof}[Proof of proposition \ref{lambda mu regular has vwb}]

\textbf{The Set up.} Let us fix a small $\epsilon>0$, and a number $\epsilon_3>0$ to be determined later. By proposition \ref{almost u-saturation property}, there is a $\xi = \xi(\epsilon(\epsilon_3)>0$ such that when $N_0'$ is large enough, we have for any $N_2 \geq N_1 \geq N_0'$ that $\bigvee_{k=N_1}^{N_2} f^k\mathcal{P}$ is an $(\epsilon_3, \xi)$ u-saturated partition.

Let $\mathcal{R} = \{ R_1, \dots, R_m\}$ be a $\epsilon_1$-regular partition by rectangles in $\Lambda^\ell$ ($\epsilon_1$ will be determined later), and for each $i=1, 
\dots, m$, let $E_i \subset R_i$ be the measurable subset on which we have 
$$
\left|\frac{d\mu_{R_i}}{d\mu_{p_i,R_i}^p} - 1 \right| < \epsilon_1,
$$
where $p_i$ is a point in $R_i$ as in the definition of $\epsilon$-regular partitions by rectangles. Recall that from the definition of $\epsilon$-regular partitions by rectangles, we have for each $i =1, \dots, m$ that $\mu_{R_i}(E_i)\geq 1-\epsilon_1$. By the construction of $\epsilon$-regular partitions by rectangles, we see that we can take the rectangles to have diameters which are as small as we want; so we assume that 
$$
\max_{i=1,\dots, m} diam(R_i)  < \frac{\xi}{100}.
$$

Since $(f,\mu)$ has the K-property, we see that there is $N_0''>0$ such that for any $N_2 \geq N_1 \geq N_0''$, we have that $\epsilon_2$-a.e. $A \in \bigvee_{k = N_1}^{N_2} f^k \mathcal{P}$ satisfies 
\begin{equation}
\left|\frac{\mu_A(R_i)}{\mu(R_i)} - 1 \right| < \epsilon_2 ,
\end{equation}
for all $i =1, \dots m+1$ (We define $R_{m+1} = M \backslash  \bigcup_{i=1}^m R_i$).  

For the rest of this section, we fix $N_0 := \max(N_0', N_0'')$, $N_2 \geq N_1 \geq N_0$. We also take the small constants $\epsilon_1, \epsilon_2, \epsilon_3 < \left( \frac{\epsilon}{100} \right)^4$. We will show that for $\epsilon$-a.e. atom $A \in \bigvee_{N_1}^{N_2}f^k\mathcal{P}$, we can construct a measurable map $\theta : A \longrightarrow M$, for which the following two properties holds
\begin{enumerate}
    \item $\theta$ is $\epsilon$-measure preserving with respect to the measures $\mu_A$ and $\mu$;
    
    \item $d(f^i(x), f^i(\theta(x))) < \epsilon$, for $i=0, 1, \dots$, for $\mu_A$-$\epsilon$-a.e. point $x \in A$. 
\end{enumerate}

This will be enough to prove the proposition, using lemma \ref{vwb charactarization for lambda mu}, and the fact that $\epsilon>0$,  $N_2 \geq N_1 \geq N_0$ are arbitrarily chosen.  
\\ \\ 

\textbf{Bad Atoms.}
Let us define the following three sets of \textit{Bad} atoms of the partition $\bigvee_{k=N_1}^{N_2}f^k\mathcal{P}$. We define $B_2$ to be the set of all those atoms $A \in \bigvee_{k=N_1}^{N_2}f^k\mathcal{P}$ for which 
$$
\left|\frac{\mu_A(R_i)}{\mu(R_i)} - 1 \right| \geq \epsilon_2,
$$
for some $i = 1, \dots, m+1$. We define $B_1$ to be the set of all atoms $A \in \bigvee_{k=N_1}^{N_2}f^k\mathcal{P}$ for which we have 
$$
\mu_A\left(\bigcup_{i=1}^m R_i\backslash E_i\right) \geq \epsilon_1^{\frac12}.
$$
We define $B_3$ to be the set of atoms of $\bigvee_{i=N_1}^{N_2}f^k\mathcal{P}$ which are not $(\epsilon_3, \xi)$ u-saturated. We call an atom $A$ of $\bigvee_{i=N_1}^{N_2}f^k\mathcal{P}$ \textit{Good} if 
$$
A \notin B_1 \cup B_2 \cup B_3.
$$
To estimate the total size of \textit{Good} atoms, we first notice that one has $\mu\left(\bigcup\limits_{A \in B_2} A\right) < \epsilon_2 $ and that $\mu\left(\bigcup\limits_{A \in B_3} A\right) < \epsilon_3 $, which follows immediately by construction. On the other hand, to estimate the size of $B_1$, we use lemma \ref{measure of bad atoms} and the straightforward fact that $\mu\left(\bigcup_{i=1}^m R_i \backslash E_i \right) < \epsilon_1$ to see that we have $\mu\left(\bigcup\limits_{A \in B_1} A\right) < \epsilon_1^{\frac12}$. Hence, the total measure of all \textit{Good} atoms of $\bigvee_{k=N_1}^{N_2}f^k\mathcal{P}$ is at least $1 -\epsilon_1^{\frac12} - \epsilon_2 - \epsilon_3$. \\ \\ 

\textbf{Bad Rectangles for Good Atoms.} Let us fix a \textit{Good} atom $A \in \bigvee_{k=N_1}^{N_2}f^k\mathcal{P}$, and define $A_i : = A\cap R_i$, for $i=1, \dots, m+1$, Notice that $A_1, \dots, A_m, A_{m+1}$ is a partition of the atom $A$. Let us also define the set $U_\xi(A) \subset A $ by  
$$
U_\xi(A) := \left \{x \in A | r^u(x) \geq \xi \mbox{ and } V^u_\xi(x) \subset A \right \}.
$$
For each rectangle $R_i \in \mathcal{R}$, set $U_{\xi,i}(A) := U_\xi(A) \cap R_i$. We define the following sets of bad rectangles for the atom $A$. We let $B_1(A)$ be the set of all rectangles $R_i$ for which $\mu_{A_i}(A_i \backslash E_i) \geq \epsilon_1^{\frac14}$. Then we let $B_2(A)$
to be the set of all rectangles $R_i \in \mathcal{R}$ for which 
$$
\mu_{A_i}(A_i\backslash U_{\xi,i}(A) ) \geq \epsilon^{\frac12}.
$$
Next, we estimate the total sizes of these bad rectangles from above, with respect to the measure $\mu_A$. Since $A$ is a \textit{Good} atom of $\bigvee_{k=N_1}^{N_2}f^k\mathcal{P}$, we see that $\mu_A\left(\bigcup_{i =1}^{m+1} R_i \backslash E_i \right) < \epsilon_1^{\frac12}$. Now, since $A_1, \dots, A_m, A_{m+1}$ is a partition of $A$, we can apply lemma \ref{measure of bad atoms}, and conclude that $\mu_A\left(\bigcup_{R_i \in B_1(A)} R_i \right) < \epsilon_1^{\frac14}$. 
On the other hand, to estimate $\mu_A\left(\bigcup\limits_{R_i \in B_2(A)}R_i \right)$, we first note that since $A$ is a \textit{Good} atom of $\bigvee_{k=N_1}^{N_2}f^k\mathcal{P}$, we have $\mu_A(A \backslash U_\xi(A)) < \epsilon_3$. Now applying lemma \ref{measure of bad atoms} with respect to partition $A_1, \dots , A_{m+1}$ of $A$, and the measure $\mu_A$, we obtain the bound $\mu_A\left(\bigcup\limits_{R_i \in B_2(A)}R_i \right) < \epsilon_3^{\frac12}$. \\ 

Now, we say that a rectangle $R \in \mathcal{R}$ is a \textit{Good} rectangle for the atom $A$, if $R \notin B_1(A) \cap B_2(A)$ and $R \neq R_{m+1}$ (We want to exclude $R_{m+1} := M \backslash \bigcup_{i=1}^m R_i$ because it is not a rectangle). From our bounds above, we see that the set of \textit{Good} rectangles for the atom $A$ has $\mu_A$-measure greater than $$
1 - \epsilon_1^{\frac14} - \epsilon_3^{\frac12} - \mu_A(R_{m+1}).
$$ 
Now, since $A$ is a \textit{Good} atom of $\bigvee_{k=N_1}^{N^2}f^k\mathcal{P}$, we know that 
$$
\left |\frac{\mu_A(R_{m+1})}{\mu(R_{m+1})} - 1 \right | < \epsilon_1.
$$
Therefore, we get the bound 
$$
\mu_A(R_{m+1}) < 2 \mu(R_{m+1}) < 2 \epsilon_1  << \epsilon_1^{\frac{1}{4}},
$$
whenever $\epsilon_1$ is small enough. Hence, we finally see that the set of all \textit{Good} rectangles for a \textit{Good} atom $A$ has $\mu_A$-measure bounded below by $1 - 2\epsilon_1^{\frac{1}{4}} - \epsilon_3^{\frac12}$.
\\

\textbf{Construction of a Map for a Good Atom in a Good Rectangle. }
Take a pair $(A, R_i)$ of a \textit{Good} atom $A$ and a good rectangle $R_i$ for $A$. We will construct a map 
$$
\theta_i : A\cap R_i \longrightarrow R_i,
$$
 which is $\epsilon_4$-measure preserving with respect to the measures $\mu_{A_i}$ and $\mu_{R_i}$ (where $\epsilon_4$ depends only on $\epsilon_1, \ \epsilon_2$ and $\epsilon_3$). First, notice that the subset $U_{xi, i}(A) \subset R_i$ is a $u$-subset of $R_i$. That is, it is saturated by full unstable manifolds $V^u_{R_i}(x)$ in $R$. Now, fix $p_i \in R_i$, and start by taking an arbitrary bijective measure preserving map 
 $$
 \theta_0 : (V^s_{R_i}(p_i) \cap U_{\xi,i}(A), \tilde \mu^s_{p_i}|_{U_i(A)}) \longrightarrow (V^s_i(p_i), \tilde \mu^s_{p_i}),
 $$
 where $\tilde \mu_{p_i}^s$ is the factor measure of $\mu$ with respect to the partition of $R_i$ by $\left(V^u_{R_i}(x) \right)_{x \in R_i}$ realized as a measure on $V^s_{R_i}(p_i)$. Such a map exists, since these are Lebesgue probability spaces with not atoms. Next, we define 
 $$
 \theta_i : U_{\xi,i}(A) \longrightarrow R_i
 $$
 by $\theta_i(x) := (\pi^u_{p_i,x} \circ \theta_0 \circ \pi^u_{x,p_i})(x) $. It is straightforward to verify that this map preserves the restriction of the product measures $\mu^{\mathcal{P}}_{p_i,R_i}|U_{\xi,i}(A)$ and the product measure $\mu^{\mathcal{P}}_{p_i,R_i}$. Now, since the pair $(A, R_i)$ is \textit{Good}, we see that $\mu_{A_i}(E_i) \geq 1-\epsilon$, and hence $\mu_{A_i}(E_i \cap U_i(A)) \geq 1 - 2\epsilon$, and in particular that $\mu_{U_{\xi,i}(A)}(E_i) \geq 1 - 2\epsilon$. On the other hand, we have 
 $$ 
 \mu^p_{R_i}(E_i) \geq \frac{\mu_{R_i}(E_i)}{1+\epsilon} \geq \frac{1-\epsilon}{1+\epsilon} > 1 - 2\epsilon.
 $$
 Hence, the conditions of lemma \ref{almost measurable map lemma} are satisfied, and therefore the map $\theta_i : U_{\xi,i}(A) \longrightarrow R_i$ is $C \epsilon$-measure preserving with respect to the measures $\mu_{U_{\xi,i}(A)}$ and $\mu_{R_i}$. We can extend this map to the whole of $A_i$, by sending each point $x \in A_i \backslash U_{\xi,i}(A)$ to itself. The resulting map $\theta_i : A_i \longrightarrow R_i$ is $C'\epsilon$-measure preserving with respect to the measures $\mu_{A_i}$ and $\mu_{R_i}$.

 Now, to construct a map
 $$
 \theta : A \longrightarrow M,
 $$
 which is $C\epsilon$-measure preserving for $\mu_A$ and $\mu$, notice that the good atom $A$ and the partition $\mathcal{R}$ satisfy the conditions of lemma \ref{glueing lemma}. Therefore, this map can be constructed by applying the lemma. 
\end{proof}

\end{section}

\begin{section}{Examples}
\textbf{Equilibrium measures for Maximal Hyperbolic Sets.}
\begin{definition}
Let $\varphi: M \longrightarrow \mathbb{R}$ be a continuous function. We say that $\mu$ is an equilibrium measure for $\varphi$ if 
$$
\sup_{\nu \in \mathcal{M}(f,M)} \left(h_\nu(f) + \int_M \varphi d\nu \right)
$$
is attained by the measure $\mu$. 
\end{definition}
These measures have been constructed for a wide variety of systems. For example, if $\Lambda \subset M $ is a maximal hyperbolic set, and $\varphi: \Lambda \rightarrow \mathbb{R}$ is a Holder continuous potential, then it is a classical result that $f|_\Lambda$ preserves an equilibrium measure for $\varphi$ (see for example \cite{sin68}, \cite{BOW} and \cite{rue76}). The classical proofs of existence of equilibrium measures in this case use either the symbolic representation of the system using Markov partitions, or the specification property of uniformly hyperbolic systems. These two approaches do not yield much information on the conditional measures of equilibrium states along unstable manifolds. Recently, the authors in \cite{CPZ1} and \cite{CPZ2} construct equilibrium states for hyperbolic attractors through a geometric approach. This approach originates in the classical works \cite{sin68} and \cite{rue76} and \cite{PS82}, where SRB measures $\mu$ where shown to satisfy
$$
\mu = \lim_{n \rightarrow \infty} \frac1n \sum_{k=0}^{n-1} f_*^k m_{V^u_{loc}(x)},
$$
where $m_{V^u_{loc}(x)}$ is the normalized Riemmanian volume. The authors in \cite{CPZ1} replace the normalized Riemmanian volume $m_{V^u_{loc}(x)}$ by a reference measure $m^{\mathcal{C}}_{V^u_{loc}(x)}$, whose construction depends on the Holder potential $\varphi$. Then, they show that 
$$
\mu = \lim_{n \rightarrow \infty} \frac1n \sum_{k=0}^{n-1} f_*^k m^{\mathcal{C}}_{V^u_{loc}(x)}
$$
is an equilibrium measure for $\varphi$. A consequence of this geometric construction is that the equilibrium measure constructed admits a local product structure! The local product strucutre for equilibrium states of axiom A systems was also obtained in \cite{L3}, where the author induces the map and the measure on a rectangle of the a Markov partition. We can apply the result of this paper to see that ergodicity, the $K$-property and the Bernoulli property also follow, without any use of Markov partitions or specification property.

\textbf{Measure of Maximal $u$-Entropy.} 
In \cite{HHW}, the authors introduce the notion of measures of maximal $u$-entropy for partially hyperbolic systems. Let $f : M \longrightarrow M$ be a $C^1$ partially hyperbolic diffeomorphism. 

\begin{definition}[Unstable Entropy \cite{HHW}]
The unstable metric entropy of $f$ with respect to $\mu$ is defined by 
$$
h^u_\mu(f) := \sup_{\eta \in \mathcal{P}^u} \sup_{\alpha \in \mathcal{P}} \limsup_{n \rightarrow \infty} \frac1n H_\mu \left ( \bigvee_{k=0}^{n-1}f^k \alpha | \eta \right )
$$
The unstable topological entropy of $f$ is defined by 
$$
h^u_{top}(f) = \lim_{\delta \rightarrow 0 } \sup_{x \in M} h^u_{top}\left(f, \overline {W^u(x, \delta)} \right),
$$
where 
$$
h^u_{top}\left(f, \overline {W^u(x, \delta)}\right) = \lim_{\epsilon \rightarrow 0} \limsup_{n\rightarrow \infty } \frac1n \log N^u(f, \epsilon, n, x, \delta).
$$
\end{definition}

\begin{definition}
An $f$-invariant measure $\mu$ is said to be of $u$-maximal entropy if $h^u_\mu(f) = h^u_{top}(f)$.
\end{definition}

in \cite{RVYY} the authors study $u$-maximal entropy measures for some partially hyperbolic extensions of some uniformly hyperbolic systems.
\begin{definition}
Let $f:M \longrightarrow M$ be a partially hyperbolic, dynamically coherent diffeomorphism $C^{1+\alpha}$ diffeomorphism, with partially hyperbolic splitting $TM = E^{cs} \oplus E^{uu}$. $f$ is said to factor over Anosov, if there is a hyperbolic linear toral automorphism $A : \mathbb{T}^d \longrightarrow \mathbb{T}^d$, and a continuous surjective map $\pi : M \longrightarrow \mathbb{T}^d$, such that 
\begin{enumerate}
    \item $\pi \circ f = A \circ \pi$
    \item $\pi$ maps each strong-unstable leaf of $f$ to an unstable leaf of $A$.
    \item $\pi$ maps each center-stable leaf of $f$ to a stable leaf of $A$.
\end{enumerate}
\end{definition}

In \cite{RVYY} the authors show that the maximal $u$-entropy measures have a local product structure for some examples, including systems that are $C^1$ close to Smale solenoids. They show that in these examples, the measures of maximal $u$-entropy are hyperbolic, and in fact, the center-stable holonomy maps act as isometries for the conditional measures. As a consequence, we see that our results above immediately apply to these classes of measures and systems.

\textbf{Local Equilibrium States} In \cite{BO22}, the author introduces the notion of \textit{local equilibrium states}
\begin{definition}[Local equilibrium state \cite{BO22}[Definition 2.12]]
Let $\varphi: WT^{\epsilon}_{\chi} \longrightarrow \mathbb{R}$ be a Grassmanian-Holder continuous potential (see \cite{BO22}[Definition 2.10]. A $\chi$-hyperbolic invariant measure is called a \textit{local equilibrium state} if
$$
h_\mu(f) + \int \varphi d\mu = \sup \left \{h_\nu(f) + \int \varphi d\nu : \nu  \mbox{ is an } f \mbox{-invariant probability measure on } H_{\chi}(p) \right\},
$$
where $p$ is hyperbolic periodic point, and $H_\chi(p)$ is an ergodic homoclinic components supporting $\mu$. 
\end{definition}
Among other things, the author shows that these local equilibrium states admit a local product structure, and hence our results apply immediately to this class of measures. 
\end{section}

\begin{section}{Appendix: Measure Theory}

This section can be considered as the measure theoretical toolbox we are going to use to prove the Bernoulli property. All the lemmas are quite general and abstract, and one might be able to apply them to establish the Bernoulli property in other contexts, for example for flows and some partially hyperbolic systems.  

\begin{definition}
Let $(X, \mathcal{B}, \mu)$ and $(Y, \mathcal{F}, \nu)$ be two probability spaces. We say that an invertible measurable map $\theta : X \longrightarrow Y$ is $\epsilon$-measure preserving if there is a measurable subset $E_\theta \subset X$ such that $\mu(E) < \epsilon$, and such that for any measurable set $S \subset X \backslash E$, we have 

$$
\left |\frac{\mu(S)}{\nu(\theta S)}-1 \right | < \epsilon.
$$
\end{definition}

The next lemma is crucial in the proof of the Bernoulli property. It allows us to establish $\epsilon$-measure preservence out of conditions on the Radon–Nikodym density. In the proof of Bernoulli, the measure being preserved is going to be the product measure on a rectangle, while the $\epsilon$-measure preserved will be the restriction of a measure with local product structure to that rectangle.
\begin{lemma}\label{almost measurable map lemma}
Let Let $(X, \mathcal{F}_X, \nu_X)$ and $(Y, \mathcal{F}_Y, \nu_Y)$ be two probability spaces. Let $\theta : (X, \nu_X) \longrightarrow (Y, \nu_Y)$ be an isomorphism of probability spaces. Let $\mu_X$ and $\mu_Y$ be two probability measures on $(X, \mathcal{F}_X)$ and $(Y, \mathcal{F}_Y)$, which are equivalent to the measures $\nu_X$ and $\nu_Y$ respectively. Let us define $\rho_X = \frac{d\mu_X}{d\nu_X}$ and $\rho_Y = \frac{d\mu_Y}{d\nu_Y}$, and let us define the the sets 
$$
E_X = \{x \in X \ | \ \left |\rho_X(x) - 1 \right | < \epsilon \},
$$

$$
E_X = \{y \in Y \ | \ \left |\rho_Y(y) - 1 \right | < \epsilon \}.
$$
If we assume that $\mu_X(E_X) \geq 1-\epsilon$, $\nu_Y(E_Y) \geq 1-\epsilon$, and that $\epsilon >0$ is small enough, then the map $\theta: (X, \mu_X) \longrightarrow (Y, \mu_Y)$ is $5\epsilon$-measure preserving.
\end{lemma}
\begin{proof}
First, let us define the measurable set  $E_\theta = E_X^c \bigcup \theta^{-1} (E_Y^c)$. We will show that 
\begin{enumerate}
    \item $\mu_X(E_\theta) < 4 \epsilon$.
    \item For any measurable set $S \subset X \backslash E_\theta$, we have
    $$
    \left | \frac{\mu_X(S)}{\mu_Y(\theta(S))} - 1 \right | < 5\epsilon.
    $$
\end{enumerate}
First, observe that 
\begin{align*}
\mu_X(E^c_X \bigcup \theta^{-1}(E^c_Y)) & < \mu_X(E^c_X) + \mu_X(\theta^{-1}(E^c_Y)) \\ & < \epsilon + \mu_X(\theta^{-1}(E^c_Y)).
\end{align*}
Now, we have 
$$
\mu_X(\theta^{-1}(E^c_Y)) = \mu_X(\theta^{-1}(E^c_Y) \cap E_X) + \mu_X(\theta^{-1}(E^c_Y) \cap E_X^c).
$$

The second term is easy to handle, since 
$$
\mu_X(\theta^{-1}(E^c_Y) \cap E_X^c) \leq \mu_X(E_X^c) < \epsilon,
$$
while to handle the first term, we observe that 
$$
\mu_X(\theta^{-1}(E^c_Y) \cap E_X) = \frac{\mu_X(\theta^{-1}(E^c_Y) \cap E_X)}{\nu_X(\theta^{-1}(E^c_Y) \cap E_X)} \nu_X(\theta^{-1}(E^c_Y) \cap E_X).
$$

Now, by our assumption $\theta$ preserves the measures $\nu_X$ and $\nu_Y$, therefore, we see that 
$$
\nu_X(\theta^{-1}(E^c_Y) \cap E_X) \leq \nu_X(\theta^{-1}(E^c_Y)) = \nu_Y(E_Y^c) < \epsilon
$$
and by our control on the density $\rho_X$ on $E_X$, we have 
$$
\frac{\mu_X(\theta^{-1}(E^c_Y) \cap E_X)}{\nu_X(\theta^{-1}(E^c_Y) \cap E_X)} < 1 + \epsilon < 2 .
$$
Putting all this together, we see that item 1 above follows. \\
Next, for $S \subset X \backslash E_\theta$ we have,

$$
\frac{\mu_X(S)}{\mu_Y(\theta(S))} = \frac{\mu_X(S)}{\nu_X(S)} \frac{\nu_X(S)}{\nu_Y(\theta(S))}\frac{\nu_Y(\theta(S))}{\mu_Y(\theta(S))} = \frac{\mu_X(S)}{\nu_X(S)} \frac{\nu_Y(\theta(S))}{\mu_Y(\theta(S))} 
$$
and therefore, we have 

$$
\left | \frac{\mu_X(S)}{\mu_Y(\theta(S))} - 1 \right | \leq \left | \frac{\mu_X(S)}{\nu_X(S)} - 1 \right | \left | \frac{\nu_Y(\theta(S))}{\mu_Y(\theta(S))} - 1 \right | + \left | \frac{\mu_X(S)}{\nu_X(S)} - 1 \right | + \left | \frac{\nu_Y(\theta(S))}{\mu_Y(\theta(S))} - 1 \right |. 
$$

Now, we can easily see that since $S \subset X \backslash E_\theta $, that 
$$
\left | \frac{\mu_X(S)}{\nu_X(S)} - 1 \right | < \epsilon
$$
and that 
$$
 \left | \frac{\nu_Y(\theta(S))}{\mu_Y(\theta(S))} - 1 \right | < (1-\epsilon)^{-1} \epsilon < 2 \epsilon.
$$
Hence 
$$
\left | \frac{\mu_X(S)}{\mu_Y(\theta(S))} - 1 \right | <  5 \epsilon.
$$
\end{proof}

The next lemma can be thought of as a "gluing lemma" for $\epsilon$-measure preserving maps. Given a partition of the space, and a subset $A$ of the space, where we have some "local" $\epsilon$-measure preserving maps from the the restriction of $A$ to the partition elements, it gives some conditions on these "local" maps that allows us to glue them together to produce an $\epsilon$-measure preserving map on the whole space. 
\begin{lemma} \label{glueing lemma}
Let $(X, \mu, \mathcal{F})$ be a probability space. Let $\mathcal{R} = \{R_1, \dots, R_m\}$ be a partition of $X$. Let $A \subset X$ be a measurable set, such that 
$$
\left |\frac{\mu(R_i)}{\mu_A(R_i)} - 1  \right |  < \epsilon.
$$
Let assume that for each $i =1, \dots, m$, we have a measurable map
$$
\theta_i : A \cap R_i \longrightarrow R_i.
$$
Let us assume furthermore that for $\mu$-$\epsilon$-a.e. $R_i$, we have that the map $\theta_i$ is $\epsilon$-measure preserving with respect to the measures $\mu_{A\cap R_i}$ and $\mu_{R_i}$. Then the measurable map
$$
\theta : A \longrightarrow X,
$$
defined by 
$$
\theta|_{A\cap R_i} = \theta_i, 
$$
is $C \epsilon$-measure preserving with respect to the measures $\mu_A$ and $\mu$, for some universal $C > 0$.

\end{lemma}
\begin{proof}
Let $B \subset \mathcal{R}$ be all the atoms of $\mathcal{R}$ for which the map $\theta_i$ is not $\epsilon$-measure preserving. Let $E_i$ be the set that appears in the definition of $\epsilon$-measure preserving maps. Let us define 
$$
E_\theta := \bigcup _{R \in B} (A\cap R) \cup  \bigcup _{i: R_i \notin B} E_i.
$$
First, let $S \subset E_\theta$, and for each $R_i \in \mathcal{R} \backslash B$, we define $S_i = S \cap R_i $, then since
$$
\frac{\mu_{R_i}(\theta_i(S_i))}{\mu_{A\cap R_i}(S_i)} = \frac{\mu(\theta_i(S_i))}{\mu_{A}(S_i)} \frac{\mu_A(R_i)}{\mu(R_i)},
$$
we see that 
$$
(1-\epsilon)^2 < \frac{\mu(\theta_i(S_i))}{\mu_{A}(S_i)} < (1+\epsilon)^2.
$$
Hence, we get
$$
(1-\epsilon)^2 < \frac{\mu(\theta(S))}{\mu_A(S)} < (1+\epsilon)^2.
$$
When $\epsilon > 0$ is small enough, we obtain
$$
\left | \frac{\mu(\theta(S))}{\mu_A(S)} - 1 \right | < 3\epsilon.
$$
Therefore, the only thing left for us to do to show that $\theta: A \longrightarrow X$ is $C\epsilon$-measure preserving is to show that $\mu_A(E_\theta) < C \epsilon$. To see this, first for any $i$ such that $R_i \in \mathcal{R}\backslash B$, we have by our assumption 
$$
\mu_{A\cap R_i}(E_i) < \epsilon,
$$
which implies that 
$$
\mu_A(E_i) < \epsilon \mu_A(R_i).
$$
On the other hand
$$
\sum_{R \in B} \mu(R)  < \epsilon,
$$
and since we have 
$$
\mu_A(R_i) < C \mu(R_i),
$$
we see that 
$$
\sum_{R \in B} \mu_A(R) < C \epsilon.
$$
Putting all of these estimates together, we get
$$
\mu_A(E_\theta) < 2C \epsilon.
$$
Hence the map $\theta : A \longrightarrow X $ is $C\epsilon$-measure preserving.
\end{proof}
We also have the following extremly simple lemma, which we will use multiple times in the proof of the Bernoulli property. 
\begin{lemma}\label{measure of bad atoms}
Let $(X, \mathcal{F}, \mu)$ be a Lebesgue probability space. Let $\eta$ be a fine measurable partition, and let $B \in \mathcal{F}$ be a measurable set, such that $\mu(B) < \delta$. Let us define 
$$
\hat B := \{ A \in \eta \ | \mu_A(B) \geq \delta^{\frac12} \}
$$
then  we have $\mu \left( \bigcup\limits_{A \in \hat B} A\right) < \delta^{\frac12}$.
\end{lemma}

\begin{lemma}\label{equivlence of product measures}
Let $(X, \mathcal{F}_X, \mu)$ and $(Y, \mathcal{F}_Y, \nu)$ be two Lebesgue probability spaces, and such that the $\sigma$-algebra $\mathcal{F}_X$ is countably generated. Let $(\mu_y)_{y \in Y}$ be a family of probability measures on $(X, \mathcal{F})$, and let us assume the following:
\begin{enumerate}
    \item The measures $\mu_y$ and $\mu$ are equivalent for $\nu$-a.e. $y \in Y$.
    \item $\int_{X} \varphi(., y) d\mu_y$ is in $L^1(Y, \nu)$, for any $\varphi \in L^1(X\times Y, \mu \times \nu) $.
\end{enumerate}

Let us define the measure $\mu_{\nu}$ by
$$
\mu_\nu(\varphi) = \int_Y \left(\int_X \varphi(x, y) d\mu_y(x)\right)d\nu(y).
$$

Then, the probability measures $\mu \times \nu$ and $\mu_\nu$ are equivalent, and we have 
$$
\frac{d\mu_\nu}{d(\mu \times \nu)}(x,y) = \frac{d\mu_y}{d\mu}(x)  ,
$$
for $(\mu \times \nu)$-a.e. $(x, y) \in X \times Y$.
\end{lemma}

\begin{proof}
First, to see that the measure $\mu_\nu$ and $\mu \times \nu$ are equivalent, let $A \subset X \times Y$ be a measurable subset, such that $A = A_X \times A_Y$, where $A_X \in \mathcal{F_X}$, and $A_Y \in \mathcal{F}_Y$. First, assume that $(\mu \times \nu)(A) = 0$, then we either have $\mu(A_X) = 0$ or $\nu(A_Y) = 0$. In the first case, we see that from the first condition in the lemma, we have $\mu_y(A_X) = 0$ for $\nu$-a.e. $y \in Y$. Now, we have 
$$
\mu_\nu(A) = \int_{A_Y} \mu_y(A_x)d\nu = 0
$$
in both cases. Therefore $\mu_n << \mu \times \nu$. The other direction is similar, if we assume that $A$ is as above, and that $\mu_\nu(A)=0$, then 
$$
\int_{A_Y} \mu_y(A_x)d\nu = 0.
$$
This implies that we either have that $\mu_y(A_x) = 0$ for $\nu$-a.e. $y \in Y$, or that $\nu(A_Y) =0$. Either case imply that $\mu \times \nu(A) = 0$. Therefore, the measures $\mu \times \nu$ and $\mu_\nu$ are equivalent. 
\\
Now, Let us define the densities $ \rho_y:= \frac{d\mu_y}{d\mu}$, and $\rho :=\frac{d\mu_\nu}{d(\mu \times \nu)}$, we want to show that 
$$
\rho(x,y) = \rho_y(x).
$$
To do that, let $\varphi \in L^1(\mu \times \nu)$, then we have by the definition of $\rho$ that 
$$
\mu_\nu(\varphi) = \int_{X \times Y} \varphi   \rho d(\mu \times \nu) = \int_{Y} \left(\int_X \varphi(x,y) \rho(x,y) d\mu(x) \right) d\nu(y),
$$
where the last equality is obtained using Fubini's Theorem. On the other hand, by the definition of the measure $\mu_\nu$, we have 
$$
\mu_\nu(\varphi)  = \int_Y \left( \int_X \varphi(x,y) d\mu_y(x)\right) d\nu(y).
$$
Therefore, we have for any $\varphi \in L^1(\mu \times \nu)$
$$
\int_{Y} \left(\int_X \varphi(x,y) \rho(x,y) d\mu(x) \right) d\nu(y)  = \int_Y \left( \int_X \varphi(x,y)d\mu_y(x)\right) d\nu(y).
$$
Let us take $\varphi(x,y) = \chi_{A_X}(x)  \chi_{A_Y}(y)$, where $\chi_{A_X}$ and $\chi_{A_Y}$ are the characteristic functions of $A_X \in \mathcal{F}_X$ and $A_Y \in \mathcal{F}_Y$ respectively. Then we see that the last equality above reduces to
$$
\int_{A_Y}\left(\int_{A_X}\rho(x,y)d\mu(x) \right) d\nu(y) = \int_{A_Y} \mu_y(A_X) d\nu(y).
$$
Now, since the set $A_Y \in \mathcal{F}_Y$ is arbitrary, and by the fact that $\mu_y(A_X)$ and $\int_{A_X}\rho(x,y)d\mu(x)$ are in $L^1(\nu)$, we see that we have the equality
$$
\mu_y(A_X) = \int_{A_X}\rho(x,y)d\mu(x),
$$
for all $y \in E_{A_X}$, where $\nu(E_{A_X}) = 1$. Now, since $\mathcal{F}_X$ is countably generated, we can find a countable base $\left(A_X^{(n)}\right)_{n \geq 1}$. Let $E = \bigcap_{n\geq 1} E_{A_X^{(n)}}$, then we see that $\nu(E) = 1$, and that for every $y \in E$, we have 
$$
\mu_y(A_X^{(n)}) = \int_{A_X^{(n)}}\rho(x,y)d\mu(x),
$$
for every $n \geq 1$. This implies that 
$$
\frac{d\mu_y}{d\mu}(x) = \rho(x,y),
$$
for all $x \in E_y$, for $\nu$-a.e. $y \in Y$, where $\mu(E_y) = 1$. Using Fubini's Theorem, we see that the equality holds for $(\mu \times \nu)$-a.e. $(x, y) \in X \times Y$, which is what we want.
\end{proof}

\begin{remark}
Under the conditions of lemma \ref{equivlence of product measures}, it is not apriori obvious if the densities $\rho_y(x) = \frac{d\mu_y}{d\mu}(x)$ are jointly measurable as a function on $X \times Y$. One consequence of the lemma is that  $\rho_y(x)$ is actually jointly measurable, which follows from $\rho_y(x) = \rho(x,y)$, for $(\mu \times \nu)$-a.e. $(x,y) \in X \times Y$, and the fact that $\rho(x,y) = \frac{d\mu_\nu}{d(\mu \times \nu)}$ is a measurable function. 
\end{remark}

\begin{lemma}\label{equiv of product measures}
Let $(X, \mathcal{F}_X, \mu)$ and $(Y, \mathcal{F}_Y, \nu)$ be two Lebesgue probability spaces, and such that the $\sigma$-algebra $\mathcal{F}_X$ is countably generated. Let us assume that we have a measurable partitions $(X_y)_{y\in Y}$ of $X \times Y$, and let $(\mu_y)_{y \in Y}$ be a family of probability measures on $(X_y, \mathcal{F}_y)$. Assume that for each $y \in Y$, there is a measurable invertible map $h_y:X_y \longrightarrow X$, such that the following holds
\begin{enumerate}
    \item The measures $(h_y)_*\mu_y$ and $\mu$ are equivalent for $\nu$-a.e. $y \in Y$.
    \item $\int_{X_y} \varphi(h_y(x), y) d\mu_y$ is in $L^1(Y, \nu)$, for any $\varphi \in L^1(X\times Y, \mu \times \nu) $.
    
\end{enumerate}
Let us define the measure $\mu_{\nu}$ by:
$$
\mu_\nu(\varphi) = \int_Y \left(\int_{X_y} \varphi(h_y(x), y) d\mu_y(x)\right)d\nu(y).
$$
Then, the probability measures $\mu \times \nu$ and $\mu_\nu$ are equivalent, and we have 
$$
\frac{d\mu_\nu}{d(\mu \times \nu)}(x,y) = \frac{d((h_y)_*\mu_y)}{d\mu}(x),  
$$
for $(\mu \times \nu)$-a.e. $(x, y) \in X \times Y$.
\end{lemma}
\begin{proof}
This is a simple corollary of the previous lemma. Let us define the measures $\mu_y' := (h_y)_* \mu_y$, these are measures on $X$, and by our assumptions, we have that $\mu_y'$ and $\mu$ are equivalent measures, for $\nu$-a.e. $y \in Y$. Next, we see that
$$
\mu_\nu(\varphi) = \int_Y \left(\int_{X_y} \varphi(h_y(x), y) d\mu_y(x)\right)d\nu(y) = \int_Y \left(\int_{X} \varphi(x, y) d\mu_y'(x)\right)d\nu(y),
$$
where the last equality holds by the definition of the push-forward measure $\mu_y'$. Now, all the assumptions of lemma \ref{equivlence of product measures} are satisfied, with respect to the measures $\mu$, $\nu$ and the family $(\mu_y')){y \in Y}$. Hence, we see that we have 
$$
\frac{d\mu_\nu}{d(\mu \times \nu)}(x,y) =\frac{d\mu_y'}{d\mu}(x)= \frac{d((h_y)_*\mu_y)}{d\mu}(x),  
$$
for $(\mu, \nu)$-a.e. $(x, y) \in X \times Y$, which is what we want.
\end{proof}
\end{section}

\begin{lemma}\label{entropy zero atomic measure}
Let $\mu$ be a hyperbolic measure and $\xi^u$ be a measurable partitions subordinate to unstable manifolds. Assume that $h_\mu(f) = 0$, then the  conditional measures $\mu_{\xi^u(x)}$ are atomic, for $\mu$-a.e. $x \in M$.
\end{lemma}
\begin{proof}
\textbf{step 1.} By \cite{LY}, we know that 
$$
h_\mu(f) = H\left(f^{-1} \xi^u | \xi^u \right),
$$
where the conditional entropy $H(\eta | \xi)$ is defined for any pair of measurable partitions $\xi < \eta$, by: 
$$
H\left( \eta | \xi\right) = - \int \log \left( \mu_{\xi(x)} \left(\eta(x) \right) \right ) d\mu(x). 
$$

Therefore, 
$$
h_\mu(f) = H\left(f^{-1} \xi^u | \xi^u \right) = -\int \log \left(\mu_{\xi^u(x)} \left(f^{-1} \left( \xi^u \left( f(x) \right)\right) \right) \right)d\mu.
$$
Now, since $h_\mu(x) = 0$, we get 
$$
\log \left(\mu_{\xi^u(x)} \left(f^{-1} \left( \xi^u \left( f(x) \right)\right) \right) \right) = 0 \Longrightarrow \mu_{\xi^u(x)} \left(f^{-1} \left( \xi^u \left( f(x) \right)\right) \right) = 1 ,
$$
for $\mu$-a.e. $x \in M$. Next, by replacing $f$ by $f^n$ for any $n \geq 1$, and observing that $h_\mu(f^n) = n h_\mu(f) = 0$, we can repeat the same argument to get
$$
\mu_{\xi^u(x)} \left(f^{-n} \left( \xi^u \left( f^n(x) \right)\right) \right) = 1 ,
$$
for all $x\in E_n$, where $E_n$ is a measurable and $\mu(E_n) = 1$.\\ 
\\
\textbf{Step 2} Define $E  = \bigcap_{n\geq 1} E_n$. Then, since $\mu(E) = 1$, we conclude that for $\mu$-a.e. $x \in M$ and for all $n \geq 1$ 
$$
\mu_{\xi^u(x)} \left(f^{-n} \left( \xi^u \left( f^n(x) \right)\right) \right) = 1. 
$$
Next, for $\mu$-a.e. point $x \in E$, there is a number $\ell_x > 0$ and an increasing sequence of times $n_1 < n_2 < \dots$, such that $f^{n_k}(x) \in \Lambda ^{\ell_x}$. let $r_0>0$ be an upper bound on the size of the elements of the partition $\xi^u$, then the size of $f^{-n_k}\left(\xi^u \left( f^{n_k}x \right) \right)$ is bounded above by $C_{\ell_x}\lambda^{-n_k}r_0$, for some $\lambda > 1$. We also know, by the Markov property of the partition $\xi^u$, that
$$
f^{-n_k}\left(\xi^u \left( f^{n_k}x \right) \right) \subset f^{-n_{k+1}}\left(\xi^u \left( f^{n_{k+1}}x \right) \right)
$$
Therefore, the set
$$
\bigcap _{k \geq 1} f^{-n_k}\left(\xi^u \left( f^{n_k}x \right) \right)
$$
can contain at most one point. Now, since we have 
$$
\mu_{\xi^u(x)} \left(\bigcap _{k \geq 1} f^{-n_k}\left(\xi^u \left( f^{n_k}x \right) \right) \right) = 1,
$$
for $x \in E$, we conclude that the measure $\mu_{\xi^u(x)}$ is atomic for $\mu$-a.e. point $x \in M$, and that it actually consists of exactly one atom.
\end{proof}
\begin{proof}[Proof of Lemma \ref{entropy zero implies periodic}]
Let $\xi^u$ be a measurable partition subordinate to unstable manifolds. Then using lemma \ref{entropy zero atomic measure}, the conditional measures $\mu_{\xi^u(x)}$ are atomic for $\mu$-a.e. point $x\in M$. Let $R$ be a hyperbolic rectangle of positive $\mu$-measure. Since the unstable manifold of $\mu$-a.e. point $x \in R$ can be partitioned by $\xi^u$ into at most countably many elements, we see that for $\mu$-a.e. point $x \in R$, the conditional measure $\mu_{V^u_R(x)}$ is atomic, with possibly a countable number of atoms. Let $x \in R$ be such a point, and let $x_0 \in V^u_R$ be an atom of $\mu_{V^u_R(x)}$. Then, using the local product structure, we get
$$
\mu_{V^u_R(y)}(\pi^s_{x,y}(x_0)) > 0,
$$
for $\mu$-a.e. point $y \in R$. In particular, this implies $\mu(V^s_R(x_0)) > 0$. Now, let us assume that we have point $x \in M$, which satisfies
$$
\mu\left(V^s_{loc}(x) \cap B_\delta(x)\right) > 0,
$$
for all $\delta > 0$. We will inductively build a sequence of nested compact sets $V_1 \supset V_2 \supset V_3 \supset \dots$, and such that
\begin{enumerate}
    \item $\mu(V_k) = \mu(V_1)$,
    \item $diam\left( V_k\right) \longrightarrow 0$ as $k \rightarrow \infty$.
\end{enumerate}
We see that the existence of such a sequence immediately implies that the measure has an atom, since $\bigcap_{k \geq 1}V_k$ consists of a single point, and has positive measure. We construct this sequence of sets inductively. Let $V_1 := V^s_{loc}(x)$, $\delta_1 = \frac12 diam(V_1)$ and $W_0 := V_1 \cap B_{\delta_1}(x)$. Then, since $\mu(W_1) > 0$, Poincare theorem implies the existence of a point $w \in W_1$ such that $f^n(w) \in W_1$, for infinitely many $n > 0$. Hence, by choosing such $n_1$ large enough, we can ensure that $f^{n_1}\left(V_1 \right)$ has a diameter as small as we want. Now, since $f^{n_k}(V_1)$ is a stable manifold of a small diameter, we conclude that it is a subset of $V_1$, since it intersects the stable manifold $V_1$ at the point $w$. Let us define $V_2 := f^{n_1}\left( V_1\right)$, and repeat the argument. So, we end up with a nested sequence of compact sets,  $V_1 \supset V_2 \supset \dots$. We also see that the two conditions are satisfied, since $V_{k+1} = f^{n_k}\left(V_k \right)$. Therefore, the measure $\mu$ is atomic, and hence it is supported on periodic points by the invariance under $f$.     
\end{proof}

\end{document}